\documentclass[11pt,letterpaper]{amsart}
\usepackage{graphicx}
 \usepackage{stmaryrd}
\usepackage{color}
\def\({\left (}
\def\){\right )}
\def\<{\left\langle}
\def\>{\right\rangle}

 \newtheorem{thm}{Theorem}[section]

\newtheorem{lem}[thm]{Lemma}

\newcommand{\abs}[1]{\left\vert#1\right\vert}

\newcommand{\R}{\mathbb R}

\newcommand{\pfrac}[2]{\frac{\partial #1}{\partial #2}}

\begin{document}
\title{ Uniqueness of the weak solutions of   the Yang-Mills flow in  four dimensions}

\numberwithin{equation}{section}

\author{Chak Hoi Chan }
\address{Department of Mathematics, The University of Queensland, Brisbane,
QLD 4072, Australia}
\email{c.h.chan@uq.net.au}

 \author{Min-Chun Hong}
\address{Department of Mathematics, The University of Queensland, Brisbane,
QLD 4072, Australia}
\email{hong@maths.uq.edu.au}

\begin{abstract}We  improve a parabolic version of Uhlenbeck's gauge fixing theorem for the Yang-Mills flow. This allows us to prove new existence of   weak solutions of the Yang-Mills flow supplying better properties than Struwe's weak solution on four manifolds in \cite {St2}. Finally, we apply the new parabolic version of Uhlenbeck's gauge fixing theorem to solve the Struwe problem  in \cite {St2} and prove uniqueness of  weak solutions to the Yang-Mills flow on compact four-manifolds with initial  value $A_0 \in W^{1, 2}(M)$.
\end{abstract}

\makeatletter
	\@namedef{subjclassname@2020}{%
	\textup{2020} Mathematics Subject Classification}
	\makeatother
 \maketitle

 \pagestyle{myheadings} \markright {Uniqueness of weak solutions to the Yang-Mills flow   }

\section{Introduction}

Let $(M,g)$ be a compact $n$-dimensional Riemannian manifold without boundary, and let $E\to M$ be a real vector bundle of rank $l$ with compact structure group $G\subseteq SO(l)$. Denote by $P\to M$ the corresponding principal $G$-bundle. The adjoint bundle is
$AdE:=P\times_{{Ad}}\mathfrak g,$
where $\mathfrak g$ is the Lie algebra of $G$. We have $AdE_{x}\simeq\mathfrak g$ (cf. \cite{U2,St2}).
For a connection $D_A$,
the Yang–Mills functional is defined by
\[
\operatorname{YM}(A;M)=\int_M |F_A|^2\,dv,
\]
where
$D_A= D_{ref}+A$ with $A\in \Omega^1(AdE)$, $D_{ref}$ is a given connection, and $F_A\in\Omega^2(AdE)$ is the curvature of $D_A$.
  A connection $D_A$ is called  Yang-Mills
if   it
satisfies the Yang-Mills equation
 \begin{equation}\label{1.1}D_A^*F_A =0\,,  \end{equation}
where $D_A^*$ is the   $L^2$-adjoint of $D_A$.

The Yang–Mills functional is invariant under a gauge transformation
$S\in G$. The Yang-Mills equation (\ref{1.1}) is not elliptic until   fixing a Coulomb gauge $S\in G$ satisfying $d^*S^*(A)=0$, so it is important to establish a gauge fixing theorem.

In   \cite {U2}, Uhlenbeck  established  the well-known gauge fixing theorem in the following:

\medskip
 \noindent{\bf Uhlenbeck's gauge fixing theorem}.  {\it   Let $D_A=d+A$ be a connection  in the  unit ball $B^n\subset \R^n$  with $A\in W^{1, n/2}(B^n)$ satisfying
 \[\int_{B^n} |F_A|^{n/2}dv \leq \varepsilon   \] for a sufficiently small $\varepsilon >0$.
 Then, there is a  gauge transformation $S\in W^{2, n/2}(B^n; G)$ such that $D_A$ is gauge-equivalent to
$D_{a}:=S^{-1}\circ D_A \circ S$   with $a=S^*(A)=S^{-1} dS+S^{-1}AS$
 satisfying
 \[d^*  a =0  \mbox{ in }B^n,  \quad  a \cdot \nu=0  \mbox { on } \partial  B^n,
 \]
  where $\nu$ is the outward normal of $\partial B^n$,  and
\begin{equation*}\label{U2}\int_{B^n}  |a|^{n/2}+|\nabla a |^{n/2}\,dv\leq C \int_{B^n} |F_{a}|^{n/2}\,dv. \end{equation*}
}

In Yang-Mills theory, the  gauge fixing theorem of Uhlenbeck in \cite {U2} has
played an important role in studying the compactification of moduli spaces of Yang-Mills connections. When $n=4$, Donaldson \cite {D1} successfully   applied   the gauge fixing theorem and  the removal  singularity theorem in \cite{U1}    to study four dimensional geometric
topology. In higher dimensions (i.e.  $n\geq 5$), Tian \cite{Ti}  applied  the gauge fixing  theorem to prove that a sequence of smooth Yang-Mills connections with a uniform $L^{2}$-bound   on  the curvature
  converges, modulo gauge transformations, to
a Yang-Mills connection in   the  smooth topology outside a closed singular set and
 that the concentration set is $(n-4)$-rectifiable. In \cite{TTi},   Tao-Tian  improved   Uhlenbeck's gauge fixing theorem   under the condition of small BMO norms
  and  proved a singularity
removal theorem for  stationary admissible Yang-Mills connections in higher dimensions (see also \cite {R}).

\medskip

In this paper, we   study the Yang-Mills flow, which was introduced by Atiyah and Bott \cite {AB}.
 The Yang-Mills flow equation is
 \begin{equation}\label{YMF} \frac {\partial A}{\partial t}=-D_A^*F_A,  \end{equation}
 with initial value $A(0)=A_0$, where $D_{A_0}= D_{ref} +A_0$ is a given
 connection.

In holomorphic vector
bundles  over K\"ahler manifolds,  Donaldson \cite {D2} applied the Yang-Mills flow to prove
  that an irreducible holomorphic vector
bundle  on K\"ahler manifolds  admits a unique
Hermitian-Einstein connection if and only if  the holomorphic vector bundle is stable. We refer to   see   further applications in \cite {Si} and \cite {Se} on the Yang-Mills flow  in holomorphic vector
bundle  on K\"ahler manifolds.

 It is interesting to  establish   existence of the Yang-Mills flow in a vector bundle $E$ over  a Riemannian manifold $M$. In lower dimensional manifolds (i.e. $n=2$ or $3$), R\aa de \cite {Ra} proved  the global existence of smooth solutions to the Yang-Mills flow in $M\times [0,\infty)$.
In the case of $n=4$,  Struwe \cite {St2}  proved existence of  weak solutions $D_A= D_{ref} +A$ of the Yang-Mills flow with initial value $A_0\in W^{1,2}(\Omega^1(M))$ in the following weak sense:
 \begin{align}&A\in C^0\left ( [0, T_1); L^2( \Omega^1(Ad \,E ))\right )\cap H^{1,2}\left ([0, T_1); L^2(\Omega^1(Ad \,E ))\right ),\label{1.3} \\&	
		 F_A\in C^0\left ( [0, T_1); L^2(\Omega^2(Ad \,E ))\right ),\\
&\int_0^{T_1}\int_M \left [\left <\frac {d A} {dt} ,\phi \right >+\left <F_A,D_A\phi  \right >\right ]\,dv\,dt=0
	\end{align}
for any $\phi\in C^{\infty}\left([0, T_1]; \Omega^1 (Ad E)\right)$ vanishing near $t=0$ and $t=T_1$, where   $T_1>0$ is  the  maximal  existence time; i.e. there exist  a uniform small $\varepsilon_0>0$ and a finite number of  singular points  $\{x_i\}_{i=1}^J\in M$ such that
\begin{equation}\label{1.6}
\limsup_{t \to T_1}  \int_{B_R(x_i)} |F_A(x, t)|^{2}\,dv \geq \varepsilon_0
\end{equation}
for any $R\in (0, R_0]$ with some $R_0>0$. Moreover, for each $t\in (0, T_1)$, the weak  solution $A(t)$ is gauge to a smooth solution of the Yang-Mills flow.

With  Casey in \cite{CCH}, we   developed  a parabolic formulation of the gauge fixing theorem for the Yang-Mills flow and proved  the maximal existence of   weak solutions to the Yang-Mills flow with  initial data $A_0\in W^{1,\frac n 2}(M)$ in higher dimension (i.e. $n\geq 4$). We must point out that our proofs in \cite{CCH} are different from the one by Struwe in  \cite {St2}.

  It is a natural question to ask whether weak solutions of the Yang-Mills flow are unique or not.
For $n\geq5$, Gastel \cite{Ga} showed non-uniqueness of the Yang–Mills flow.
 When $n=4$, Struwe \cite {St2} proved   uniqueness of  weak solutions of the Yang-Mills flow under an extra  condition that $A_0$ is irreducible; i.e. for all $s\in \Omega^0(Ad E)$,
\begin{equation}
\|s\|_{W^{1,2}(M)}\leq C \|D_{A_0} s\|_{L^2(M)}.\label{IRD}
\end{equation}
In general case of vector bundles $E$ over $M$, the irreducible condition  \eqref{IRD} are not satisfied. Therefore,
 it has been  an open problem for 30 years about the uniqueness of  weak solutions of the Yang-Mills flow in four manifolds on the initial data  $A_0\in W^{1,\,2}(M)$ without the irreducible condition \eqref{IRD}. This problem was again mentioned in \cite {W}.

In this paper, we solve the above uniqueness problem for weak solutions of the Yang-Mills flow   and  prove
\begin{thm}\label{Theorem 1.2} When $n=4$, let $D_{A_0}$  be a connection with     $A_0\in W^{1,\, 2}(M)$. Then,   the  weak solutions,    which were obtained by Struwe in \cite{St2},    of the Yang-Mills flow (\ref{YMF}) with initial value $A_0\in W^{1, \,2}(M)$ are unique. \end{thm}

Let $A_0$ be an initial connection satisfying  $A_0\in W^{1, 2}(M)$ satisfying
 \begin{equation*}
\sup_{x\in M} \int_{B_{r_0}(x)}|F_{A_0}|^2\,dv\leq \varepsilon_0
 \end{equation*} for a sufficiently small $\varepsilon_0$ and some $r_0>0$. For the compact manifold $M$,   there are a finite number of points $y_i\in M$, $i=1,\cdots,L$,  such that $\{U_i=B_{r_0}(y_i)\}_{i=1}^J$ covers $M$.  Moreover, for each $x\in M$ there are at most finite number $K$  of nonempty intersections of  different  balls $U_j$.  By Uhlenbeck's gauge fixing theorem,  for each  ball $U_i$, there exist    a   gauge
transformation  $S_i(0)\in W^{2,2}(U_i; G)$    and  a  connection
$D_{a_i(0)}=S_i^*(0)(D_{A_0})=d+a_i(0)$ such that
 \begin{equation} \label{IU}  d^*  a_i(0)=0 \quad \mbox {in } U_i,\quad a_i(0)\cdot \nu=0 \mbox { on }\partial  U_i. \end{equation}
    In order to prove Theorem \ref{Theorem 1.2},  we need to establish a parabolic   gauge fixing theorem for weak solutions of  the Yang-Mills flow    with non-smooth initial data $A_0\in W^{1, 2}(M)$. However, the parabolic  gauge fixing theorem for the Yang-Mills flow  in \cite{CCH} is not good enough to prove uniqueness of weak solutions.
In order to overcome the difficulty, we apply  Gaffney's inequality (see, e.g., \cite{G1}, \cite{F}) to improve  the parabolic gauge fixing result  in \cite{CCH}  by  establishing new existence of weak solutions with better properties in the following:

\begin{thm}\label{Theorem 4.2}  When $n=4$, let $D_{A_0}$  be a connection with   $A_0 \in W^{1, 2}(M)$. Then
there exists a local weak solution $A(t)$ of the Yang-Mills flow (\ref{YMF}) in $M\times [0, t_1]$ with initial connection $A_0$ with some $t_1>0$  satisfying
 \begin{equation}\label{6.1}
\sup_{0\leq t\leq t_1} \int_{U_i}|F_{A }(\cdot, t)|^2\,dv\leq \varepsilon.
 \end{equation}
Then,
   for each  ball $U_i$, there exist    a gauge
transformation  $S_i(t)\in W^{1,2}(U_i; G)$    and  a  connection
$D_{a_i}=S_i^*(D_{A})=d+a_i$ such that
 \begin{equation}\label{1.9}  d^*  a_i(t)=0 \quad \mbox {in } U_i,\quad a_i(t)\cdot \nu=0 \mbox { on }\partial  U_i \end{equation} for each $t\in [0, t_1]$  and \begin{equation}\label{1.11}
    \sup_{0\leq t\leq t_1}
    \int_{U_i}r_0^{-2} \abs{a_i(t)}^2 +\abs{\nabla  a_i(t)}^2 dv \leq C\varepsilon.
    \end{equation}
 Moreover, $D_{a_i}$  is a smooth solution of the equation
  \begin{equation}\label{1.10}
    \pfrac{a_i}{t}=-D_{a_i}^*F_{a_i}+D_{a_i} s_i\quad \mbox{in $ U_i\times (0,t_1]$,}
     \end{equation}
     where
      \begin{equation}\label{s_i}s_i(t)=S_i^{-1}(t)\circ \frac d{dt} S_i(t)\in W^{1,2}(U_i), \quad \int_{U_i}s_i(x,t)\,dv=0.  \end{equation}
  Moreover, we have
  \begin{equation}\label{1.12}
  \begin{aligned}
  \int_0^{t_1}\int_{U_i}\left|\frac{\partial a_i}{\partial t}\right|^2+ \abs{\nabla^2a_i }^2 +|\nabla s_i|^2dvdt  \leq C\left(1+\frac {t_1}{r_0^2}\right)\int_M|F_{A_0}|^2\, dv.
  \end{aligned}
  \end{equation}
  \end{thm}
We would like to point out that two estimates \eqref{s_i}-\eqref{1.12}  has improved the result in \cite{CCH} and  will play a key role in proof of  Theorem \ref{Theorem 1.2}.

 Second key idea of proving Theorem \ref{Theorem 1.2} is to apply Theorem \ref{Theorem 4.2} to   convert boundary estimates $a_i $ on $\partial U_i$ to the interior estimates of $a_j$ in $U_j$ by using the transition behaviors of gauge transformations in $U_i\cap U_j$.
 Let $a_i^k$ with $k=1, 2$ be two weak solutions satisfying \eqref{1.9}-\eqref{1.12}. In the nonempty intersection $U_i\cap U_{j}$, there exist  gauge transformations $S_{ij; k}(t)$  (see Lemma 3.5 of \cite{U2} or \cite{Se})   such that
\begin{eqnarray*}
 a^k_{j}(t)= {S}^{-1}_{ij; k}(t)  d S_{ij; k}(t) + S^{-1}_{ij}(t) a^k_{i}(t)S_{ij; k}(t).
 \end{eqnarray*}
 Then, we apply a delicate proof to estimate the difference of gauge   transformation $S_{ij,2}(t)-S_{ij,1}(t)$ and    curvatures $F_{a_i^2(t)}-F_{a_i^1(t)}$ in $U_i\cap U_{j}$, and obtain
 \begin{equation*}
\mathcal E(t)
\le C\left(1+tK(t)\right)
\int_0^t h(\tau)\mathcal E(\tau)\,d\tau,
\end{equation*}
where $\mathcal E(t)
:=\sum_i\int_{U_i}
|F_{a_i^2(t)}-F_{a_i^1(t)}|^2\,dv$,   $h(t)\in L^1[0,t_1]$ and  $tK(t)\in L^\infty(0,t_1)$ (see details in Section 4).  Then, Gronwall's inequality implies the  uniqueness.

The paper is organized as follows. In Section 2, we present some necessary   and basic estimates on the Yang-Mills flow. In Section 3, we improve the parabolic gauge fixing theorem and prove Theorem \ref{Theorem 4.2}. In Section 4, we prove Theorem \ref{Theorem 1.2}.

\medskip
\noindent{\bf AI Disclosure:}{
The authors used GPT $6.1$-Sol as a tool for improving some mathematical  proofs only in Section 4. The main idea of this paper comes  from two authors, but is not designed by AI. The authors    independently developed and verified all proofs, and   take full responsibility for the content of the manuscript. After announcing in arXiv, Theorem \ref{Theorem 1.2} was formalized in Lean 4.
}

\section{Preliminary results}

In this section, we assume that $n=4$. At first, we  recall some notations and results about the Yang-Mills  flow. For an integer $k\geq 0$, set   $\Omega ^k(Ad \,E ) = \Gamma
(Ad \,E \otimes\Omega ^k(M))$.
In local coordinates,   $a\in \Omega ^k(Ad \,E )$ is  represented   by
$$  a(x) =  \sum_{i_1<\cdots <i_k} a_{i_1, \cdots, i_k}(x)dx_{i_1}\wedge \cdots  \wedge dx_{i_k} $$
with $a_{i_1, \cdots, i_k}(x)\in \mathfrak{g}$.
For $a, b\in \Omega ^k(Ad \,E )$, their inner product $\<a,b\>$  is defined by
$$  \< a, b\> =
    \sum \text{Trace}\(a_{i_1,\cdots, i_k}b^T_{j_1, \cdots, j_k}\)
     \< dx_{i_1}\wedge  \cdots \wedge  dx_{i_k},
      dx_{j_1}\wedge \cdots  \wedge  dx_{j_k}\> , $$
where $b^T_{j_1, \cdots, j_k}$ is the  transpose of   $b_{j_1, \cdots, j_k}$.

Each connection  $D_A=d+A$ is  compatible with the metric structure of $E$. For an exterior derivative
$$D_A : \Omega ^k(Ad\,E) \to \Omega ^{k+1}(Ad \, E) , \;\;  $$
there is the adjoint operator
$$D_A^*: \Omega^{k+1}(Ad E )\to \Omega^{k}(Ad E)$$
satisfying the property
$$ \int _M \<D_Aa, b\>\,dv=\int _M \<a, D_A^*b\>\,dv$$
for all $a\in \Omega^{k}(Ad E)$ and $b\in \Omega^{k+1}(Ad E )$.
With the above fibrewise inner product,   we  can define
$L^2\(\Omega ^p(Ad \,E )\)$ as the completion of $C^\infty $ sections
under these $L^2$ norms.

We recall two well-known   estimates on the Yang-Mills flow (e.g. see,      \cite{St2}):

\begin{lem}  \label{Lemma 1} Let $A(t)$ be a smooth solution of the Yang-Mills flow (\ref{YMF}) in $M\times
 [0,T]$ with initial connection $A(0)=A_0$ for some $T>0$. Then, for $0<s\leq T$, we have
\begin{equation} \label{Energy identity}
  \int_M |F_{A(s)}|^2  \,dv +2
  \int_0^s \int_M  \abs{\pfrac{A}{t}}^2 dv\,dt= \int_M  |F_{A_0}|^2
  \,dv
\end{equation}
and
 \begin{equation} \label{Loca identity}
    \int_{B_{R_0}(x_0)}| F_{A(s)}|^2 \, dv\leq \int_{B_{2R_0}(x_0)} |F_{A_0}|^2
  \,dv + \frac C {R_0^2}\int_0^s  \int_{B_{2R_0}(x_0)} |F_A|^2 \, dv\,dt.
 \end{equation}
\end{lem}

\begin{lem}  \label{Lemma 2} Let $A(t)$ be a smooth solution of the Yang-Mills flow (\ref{YMF}) in $M\times
 [0,T]$ with initial connection $A(0)=A_0$. There exist a sufficiently small constants  $\varepsilon>0$ and $0 < R_0 < 1$ such that if
 \[\sup_{0\leq t\leq T} \max_{x_0\in M} \int_{B_{R_0}(x_0)} |F_A(\cdot ,t)|^{2}dv \leq \varepsilon \]
 Then there exists a constant $C > 0$ such that
 \begin{equation}\label{T2}
    \int_0^T \int_M|\nabla_A F_A|^2\,dv\, dt\leq C\left(1+\frac {T}{R_0^2}\right)\int_M|F_{A_0}|^2\, dv.
\end{equation}
\end{lem}

We have estimates of higher order derivatives (e.g., see, \cite{HT1}, \cite{HTY}, \cite{CCH}):
\begin{lem}  \label{Lemma 6}   Let $A(t)$ be a solution to the Yang-Mills flow (\ref{YMF}) in $M\times
 [0,T]$ with initial connection $A(0)=A_0$.   Assume that there is a sufficiently small constant $\varepsilon>0$ such that
\begin{eqnarray}\label{A-G}
\sup_{ t_0-4R_0^2\leq t\leq t_0}  \int_{B_{2R_0}(x_0)} |F_A(x,t)|^{2}dv \leq \varepsilon
\end{eqnarray}
 for any $x_0\in M$ and $t_0>4R_0^2$. Then   for any integer $l\geq 1$, we have
 \begin{eqnarray} \label{2.28'}
&&\sup_{t_0-R_l^2\leq t\leq t_0}
\int_{B_{R_l}(x_0)}
\sum_{i=1}^{l}
|\nabla_A^iF_A|^{\frac{2l}{i}}\,dv
+
\int^{t_0}_{t_0-R^2_l}\int_{B_{R_l}(x_0)}
|\nabla_A^{l+1}F_A|^2\,dv\,dt\\
&
\leq&
C_l\sum_{i=1}^{l}
R_0^{\,4-2l-\frac{4l}{i}},\nonumber
\end{eqnarray}
   for some constant  $C_l > 0$ independent of  $A$ and $R_0$, where  $R_l:=\frac 12 R_0+\frac {R_0}{2^{l+1}}$ for  each $l\geq 1$.
\end{lem}

We  recall two estimates from Uhlenbeck \cite{U2}, which are related to Gaffney's inequality (see, e.g., \cite{F}, \cite{G1}).

\begin{lem} \label{L4}Let $U=B$ be the unit open ball in $\R^4$ and let $A\in W^{1, p}(U)$ with $2<p<4$. Assume that  $d^*A=0$ in $U$ and $A\cdot \nu =0$ on $\partial U$. Then
\[\int_U |\nabla A|^2dv + \int_{\partial U} |A|^2d\omega = \int_U |d A|^2dv. \]
\end{lem}

\begin{lem} \label{L5} Let $U=B$ be the unit open ball in $\R^4$. Let  $D_A$ be a smooth connection satisfying   $d^*A=0$ in $U$ and $A\cdot \nu =0$ on $\partial U$. Then, for any $2\leq p<4$, we have
\begin{equation}\label{3.F.0}
\|A\|_{W^{1,p}(U)} \leq C \left (\int_U |d A|^p\,dv\right )^{1/p} \end{equation}
for some constant $C$ independent of $A$. Moreover, there exists a sufficient small $\varepsilon$ such that if
 \begin{equation}\label{3.F.1}
 \int_{U}|A|^4 dv \leq C\varepsilon,
 \end{equation}
 then
  \begin{equation}\label{3.F.2}
  \|A\|_{W^{1,p}(U)} \leq C \left (\int_U |F_A|^p\,dv\right )^{1/p}.
 \end{equation}
\end{lem}
\begin{proof} \eqref{3.F.0} was mentioned in \cite{U2} without   giving details. For completeness, we give detailed proofs here.

 Note $d^*A=0$ in $U$.
Then
\[\Delta  A=-d^*dA.\]
Let $U'\subset\subset U$, the local interior $L^p$-estimate yields
\begin{equation}\label{2.11}
\int_{U'} |\nabla A|^p\,dv \leq C_{U'}\int_{U}|dA|^p+ |A|^{p}\,dv.
\end{equation}
 Without loss of generality, we can assume  for $y_i\in \partial U$
\[\partial U\cap B_{\rho}(y_i)\simeq\{x_n=0\}\cap B_{\rho}(y_i). \]
Note  $A\cdot \nu =A_4=0$ on $\partial U$.  Then we  extend
\begin{align*}
	 \bar A(x)=A(x)=&A_1(x_1, x_2, x_{3}, x_4) dx_1+\cdots +A_{3} (x_1, x_2, x_{3} , x_4)dx_{3}\\
& +A_4 (x_1, x_2, x_3 , x_4) dx_4  \end{align*}
in $B_{\rho}^+(y_i)=B_{\rho}(y_i)\cap \{x_4\geq 0\}$
to
  \begin{align*} \bar A(x)= &A_1 (x_1, x_2, x_{3}, -x_4) dx_1+\cdots + A(x_1, x_2, x_{3}, -x_4)dx_{3}\\ &-A_4  (x_1, x_2, x_{3}, -x_4) dx_4 \end{align*}
  in $B_{\rho}^-(y_i)= B_{\rho} (y_i)\cap \{x_n\leq 0\}$.
Then it can be checked that   $\bar A\in W^{1,p}(B_{\rho})$.
 and
\[\int_{B^-(y_i)}|d\bar A|^p\,dv= \int_{B^+(y_i)}|d\bar A|^p\,dv, \quad \int_{B^-(y_i)}|d^*\bar A|^p\,dv= \int_{B^+(y_i)}|d^*\bar A|^p\,dv.\]

Using the local interior $L^p$-estimate again to the equation $\Delta \bar A=-d^*d\bar A-d d^*\bar A$, we have

\begin{align} \label{Bd}\int_{B_{3/4}(y_i)\cap U} |\nabla \bar A|^p\,dv \leq C \int_{B(y_i)\cap U}|d\bar A|^p+|d^*\bar A|^p+ |\bar A|^{p}\,dv.   \end{align}

After a transformation, $d^*A=0$ in $U\cap B_{\rho}(y_i)$ becomes
\[
d^* \bar {A}=-g^{ij}\partial_i  {A}_j+g^{ij}\Gamma^k_{ij} {A}_k=0  \mbox { in $B_{\rho}^+(y_i)$,}\]  where $\Gamma^k_{ij}$ denote the Christoffel symbol of a metric $g^{ij}$ on  $\partial U\cap B_{\rho}(y_i)$. Since $g^{ij}(y_i)=\delta_{ij}$, by choosing $\rho$ sufficiently small, we obtain
\begin{equation}
 \|d^*\bar{A}\|_{L^p(B_\rho^+)}
\le  \delta \|\nabla \bar{A}\|_{L^p(B_\rho^+)}
+\frac C{\rho}\|\bar{A}\|_{L^p(B_\rho^+)}
\end{equation} for a small $\delta$.
Then
\begin{equation}\label{2.13}
\int_{B_{\rho/2}^+}|\nabla \bar A|^p\,dv
\le C\int_{B_\rho^+}\bigl(|d\bar A|^p+\delta |\nabla A|^p+\frac 1 {\rho^p}|\bar A|^p\bigr)\,dv.
\end{equation}
We can cover $\bar{U}$ by $\bar U \subseteq \cup_{i=1}^L B_{\frac{\rho}{2}}(y_i)\cup U'.$
 Summing the interior estimate \eqref{2.11} and the boundary estimate \eqref{2.13} for  sufficiently small $\delta$, we obtain
 \begin{equation}\label{2.14.1}  \int_{U} |\nabla A|^p\,dv \leq C\int_{U}|dA|^p+ |A|^{p}\,dv.  \end{equation}
If  $A\in W^{1, p}(U)$ satisfying $A\cdot \nu =0$ on $\partial U$,  we claim
 \begin{equation}\label{2.14.2} \int_{U} |A|^{p} dv \leq C\int_{U}|dA|^p+|d^*A|^p \,dv.  \end{equation}
In fact, if not, there is a sequence $\{A^k\}_{k=1}^{\infty}\subset W^{1,p}(U)$ such that   $A^k\cdot \nu =0$ on $\partial U$ satisfying
\[\int_{U} |A^k|^{p}\,dv =1, \quad \int_{U} |d A^k|^{p}+|d^*A^k|^p\,dv\leq \frac 1 k.  \]
As $k\to\infty$, $A^k$ converges weakly to an  one-form $A$ in $W^{1,p}(U)$ satisfying
\[\int_{U} |A|^{p}\,dv =1, \quad \int_{U} |d A |^{p}+|d^*A|^p\,dv=0.\]
It implies that $A\in W^{1, p}(U)$ with $2\leq p<4$   satisfies $dA=0$ and
$d^*A=0$ in $U$ with $A\cdot \nu =0$ on $\partial U$.  Then, using Lemma 2.4, $A=0$ on $\partial U$, which implies $A=0$ in $U$. This is a contradiction.
Therefore, we prove \eqref{2.14.2}.
Combining \eqref{2.14.1} with \eqref{2.14.2} yields \eqref{3.F.0}.

Applying H\"older's inequality and Sobolev inequality, we get
\begin{eqnarray*}
\int_U |A|^{2p}\,dv
&\leq&
\left(\int_U |A|^4\,dv\right)^{p/4}
\left(\int_U |A|^{\frac{4p}{4-p}}\,dv\right)^{\frac{4-p}{4}}\\
&\leq&
C\left(\int_U |A|^4\,dv\right)^{p/4}
\int_U \bigl(|A|^p+|\nabla A|^p\bigr)\,dv.
\end{eqnarray*}
Since $dA=F_A-A\wedge A$, it follows from \eqref{3.F.0} and \eqref{3.F.1}  that
\begin{eqnarray*}
&&\int_U |A|^p+|\nabla A|^p\,dv
\leq
C\int_U |F_A|^p\,dv+C\int_U |A|^{2p}\,dv\\
%&\leq&C\int_U |F_A|^p\,dv
%+C\left(\int_U |A|^4\,dv\right)^{p/4}
%\int_U \bigl(|A|^p+|\nabla A|^p\bigr)\,dv\\
&\leq&
C\int_U |F_A|^p\,dv
+C\varepsilon^{p/2}
\int_U \bigl(|A|^p+|\nabla A|^p\bigr)\,dv.
\end{eqnarray*}
Choosing $\varepsilon$  sufficiently small, we prove \eqref{3.F.2}.
\end{proof}

We remark that  Lemma \ref{L5}  also holds on a small geodesic ball $U=B_{r_0}$ on $M$ for a sufficiently small $r_0$; i.e.
let  $A$ be a  connection in $W^{1,p}(U)$ satisfying   $d^*A=0$ in $B_{r_0}$ and $A\cdot \nu =0$ on $\partial  B_{r_0}$. If there exists a sufficient small $\varepsilon$ such that
 \begin{equation*}
 \int_{B_{r_0}}|A|^4 dv \leq C\varepsilon,
 \end{equation*}
 then
  \begin{equation*}
 \int_{B_{r_0}}\left (\frac  1  {r_0^p}  |A|^p+|\nabla A|^p\right ) \,dv \leq C \int_{B_{r_0}} |F_A|^p\,dv.
 \end{equation*}
 When $r_0$ is sufficiently small, the geodesic ball $B_{r_0}$ on $M$ is closed to a flat ball in $\R^4$, so we have
\[|d_g A-dA| +|d^*_g A-d^*A|\leq \delta |\nabla A| +C|A|\]
for a sufficiently small $\delta >0$. Then \eqref{2.14.1} holds.

\section{A parabolic gauge fixing theorem and existence of  weak solutions}

Since $M$ is compact, there is a finite  cover of open balls $\{U_i\}_{i=1}^L$ of $M$ with $U_i=B_{r_0}(x_i)$, such that at each $x\in M$, there are at most a finite number $K$ of  balls intersecting.
In \cite{CCH}, we develop a parabolic formulation of the gauge fixing theorem of the Yang-Mills flow.
Now, we improve the  parabolic   gauge fixing theorem  by

 \begin{lem}\label{Lemma 3.2} Let $D_{A(t)}= d +A(t)$ be a smooth   solution  of the Yang-Mills
flow (\ref{YMF}) in $M\times [0,t_1]$ for some $t_1>0$ with  smooth initial value  $A_0$ satisfying
 \begin{equation}\label{3.1.1}
\sup_{0\leq t\leq t_1} \int_{U_i}|F_{A }(\cdot, t)|^2\,dv\leq \varepsilon.
 \end{equation}
Then,
   for each  ball $U_i$, there exist    a smooth gauge
transformation  $S_i(t)$    and  a  connection
$D_{a_i}=S_i^*(D_{A})=d+a_i$ such that
 \begin{equation}\label{3.2.1}  d^*  a_i(t)=0 \quad \mbox {in } U_i,\quad a_i(t)\cdot \nu=0 \mbox { on }\partial  U_i \end{equation} for each $t\in [0, t_1]$ and
 \begin{equation}\label{3.4.1}
    \sup_{0\leq t\leq t_1}
    \int_{U_i}r_0^{-2} \abs{a_i(t)}^2 +\abs{\nabla  a_i(t)}^2 dv \leq C\varepsilon.
    \end{equation}
  Moreover, $D_{a_i}$  is a smooth solution of the equation
  \begin{equation}\label{3.3.1}
    \pfrac{a_i}{t}=-D_{a_i}^*F_{a_i}+D_{a_i} s_i\quad \mbox{in $ U_i\times (0,t_1]$,}
     \end{equation}
     where \[s_i(t)=S_i^{-1}(t)\circ \frac d{dt} S_i(t)\in W^{1,2}(U_i), \quad \int_{U_i}s_i(x,t)\,dv=0,   \]
satisfying
  \begin{equation}\label{3.5.1}
  \begin{aligned}
  &\int_0^{t_1}\int_{U_i}\left|\frac{\partial a_i}{\partial t}\right|^2+ \abs{\nabla^2a_i }^2 +|\nabla s_i|^2dvdt\\
   \leq &C \int_0^{t_1}\int_{U_i} (|\nabla_{a_i}F_{a_i}|^2+\frac 1{r^2_0} |F_{a_i}|^2 )\,dv\,dt \leq C\left(1+\frac {t_1}{r_0^2}\right)\int_M|F_{A_0}|^2\, dv.
  \end{aligned}
  \end{equation}
\end{lem}
\begin{proof}
 Using \eqref{3.1.1}, we apply Uhlenbeck's gauge fixing theorem in our previous paper \cite{CCH}   to obtain   that for each $U_i$, there exists smooth gauge
transformations $\tilde S_i(t)\in G$  such that $D_A$ is gauge-equivalent to
$D_{\tilde a_i (t)}=\tilde S_i^*(D_{A(t)})$ in $U_i$, which satisfies
 \begin{equation}\label{US2}  d^* \tilde a_i(t)=0 \quad \mbox {in } U_i,\quad  \tilde a_i(t)\cdot \nu=0 \mbox { on }\partial  U_i \end{equation}
 and
 \begin{equation}\label{US2}  \int_{U_i} \frac {| \tilde a_i(t)|^{2}}{r^{2}_0}+|\nabla  \tilde a_i(t)|^{2}\,dv\leq C   \int_{U_i}|F_{ \tilde a_i(t)}|^{2}\,dv\leq C\varepsilon_0= \varepsilon. \end{equation}
Moreover, it was shown in \cite{CCH} that $D_{a_i}$ is a smooth solution to the equivalent Yang-Mills flow
\begin{eqnarray}\label{good}
    \pfrac{ \tilde a_i}{t}=-D_{\tilde a_i }^*F_{\tilde a_i }+D_{\tilde a_i } \tilde s_i\quad\mbox { in } {U_i}\times [0,t_1]
    \end{eqnarray}
with $\tilde s_i(x,t)=\tilde S_i^{-1}(x, t)\circ \frac d{dt} \tilde S_i(x, t)$.

By the
theory of ordinary differential equations, there exists a unique smooth
solution $S_{2,i}(t)\in G$ to the following initial value problem:
 \begin{equation}\label {3.9} \frac d{dt} S_{2,i}=-\left (\frac 1{|U_i|}\int_{U_i} \tilde s_i\,dv \right )\circ S_{2,i} \quad \mbox{in } [0, t_1]
\end{equation}
with $S_{2,i}(0)=I$. Set $D_{a_i}=S_{2,i}^*(D_{\tilde a_i})$, which satisfies
 \begin{eqnarray}\label{EY}
    \pfrac{a_i}{t}=-D_{a_i}^*F_{a_i}+D_{a_i}s_i\quad\mbox{ in }U_i\times [0,t_1]
    \end{eqnarray}
    with $s_i(x,t)=S_{2,i}^{-1}(t)\circ \frac d{dt} S_{2,i}(t)+S_{2,i}^{-1}(t)\tilde s_i(x,t)S_{2,i}(t)$. It can be seen from \eqref{3.9} that
    \[\int_{U_i} s_i(x,t)\,dv
    =S_{2,i}^{-1}(t)\circ \frac d{dt} S_{2,i}(t)|U_i|
    +S_{2,i}^{-1}(t)\int_{U_i}\tilde s_i(x,t)\,dv\,S_{2,i}(t)=0. \] Furthermore, $D_{a_i}$ satisfies  \eqref{3.2.1} and \eqref{3.4.1}.

Next, we will prove \eqref{3.5.1}.
Using \eqref{EY},  we obtain
   \begin{equation} \label{3.11}  \int_{0}^{t_1 }  \int_U| \pfrac{a_i}{t}- D_{a_i} s_i  |^2\,dv\,dt\leq C \int_{0}^{t_1 }  \int_U|D^*_{a_i}
    F_{a_i}|^2\,dx\,dt.  \end{equation}
Using \eqref{3.2.1}, we have
\begin{align}\label{3.12}
		&  \int_{U_i}  \left <\pfrac{a_i}{t} , \, D_{a_i} s_i\right >\,dv = \int_{U_i}  \left <\pfrac{d^* a_i}{t} , \, s_i\right >\,dv+ \int_{U_i}  \left <\pfrac{a_i}{t} , \, [a_i, s_i]\right >\,dv
		\\
		& \leq \frac 12  \int_{U_i} |\pfrac{a_i}{t}|^2\,dx + C(\int_{U_i}   |a_i|^4 \,dv)^{1/2} (  \int_{U_i}   |s_i|^4\,dv)^{1/2}\nonumber
\\
		& \leq \frac 12  \int_{U_i} |\pfrac{a_i}{t}|^2\,dv + C\varepsilon^{1/2}    \int_{U_i}   |\nabla s_i|^2\,dv.\nonumber
	\end{align}
Combining  \eqref{3.11} with  \eqref{3.12} and choosing $\varepsilon$ sufficiently small, we obtain
   \begin{equation} \label{3.16}
   \begin{aligned}
    \int^{t_1}_0 \int_{U_i}| \pfrac{a_i}{t}|^2 +|\nabla s_i|^2\,dv\,dt&\leq C \int_{0}^{t_1 }  \int_{U_i}  |D^*_{a_i}
    F_{a_i}|^2\,dvdt\\
    &\leq C\int^{t_1}_0 \int_{U_i}|\nabla_{a_i}F_{a_i}|^2 dvdt.
   \end{aligned}
   \end{equation}
Using Lemma \ref{L5} and Sobolev's inequality, we have
\begin{equation}\label{3.14}
\begin{aligned}
&r_0^{-1}\|a_i\|_{L^2(U_i)}
+\|\nabla a_i\|_{L^2(U_i)}
+\|a_i\|_{L^4(U_i)}\le C\|F_{a_i}\|_{L^2(U_i)}
\le C\varepsilon^{1/2}.
\end{aligned}
\end{equation}
Recalling Lemma \ref{Lemma 2}, we note
    \begin{equation}\label{l2} \int_0^{t_1} \int_{U_i} |\nabla_{a_i} F_{a_i}|^2\,dv\, dt\leq C\left(1+\frac {t_1}{r_0^2}\right)\int_M|F_{A_0}|^2\, dv.
\end{equation}
Without loss of generality, we can assume  that
for $x_0\in \partial U_i$ and $r_0=1$,
\[\partial U_i\cap B_{1/2}(x_0)= \{x_4=0\}\cap B^+_{1/2} . \]
and    $d^*a=0$ in $B^+_{1/2}$ (see details in Lemma \ref{L5}).

Denote $\tilde x=(x_1, x_2, x_{3})$ and $\tilde \nabla =(\partial_{x_1}, \cdots , \partial_{x_{3}})$.

 In $B^+$, $d^* \tilde \nabla a=0$ in $B_{1/2}^+$ and $ \tilde \nabla a\cdot \nu=0$ on $\{x_4=0\}\cap B^+_{1/2}$.
 Let $\phi\in C_0^{\infty}(B_1)$ be a cut-off function such that $0\leq \phi\leq 1$, $\phi=1$ on $B_{\frac{1}{2}}$ and $\nabla \phi \leq C$.
 Through  Gaffney's inequality    (see, e.g., \cite{G1}, \cite{F}), we note
\begin{eqnarray}\label{GD0}
 &&\int_{B_1^+} |\nabla (\tilde \nabla a_i \phi )|^2 \,dv \leq C \int_{B_1^+} |d (\tilde \nabla a_i  \phi) |^2 + |d^* (\tilde \nabla a_i  \phi) |^2+|\nabla a|^2  \phi^2  \,dv\\
 &\leq & C   \int_{B_1^+} |d \tilde \nabla a_i )|^2 \phi^2+|\nabla a|^2 (|\nabla \phi |^2+ \phi^2)  \,dv.\nonumber
\end{eqnarray}
 Note
   \[\nabla_{a_i} F_{a_i}=  \nabla da_i+ a_i\#\nabla a_i +a_i\#a_i\#a_i.\]
  Using \eqref{GD0}, we have
   \begin{eqnarray}\label{GD1}
&& \int_{B_1^+} |\tilde \nabla d a_i|^2\phi^2\,dv\leq C \int_{B_1^+} |\nabla (\tilde \nabla a_i \phi )|^2 +|\nabla a|^2 |\nabla \phi|^2\,dv  \\
 &\leq &  2  \int_{B_1^+} |\nabla_{a_i} F_{a_i}|^2\phi^2\,dv  + C \int_{B_1^+}  (|a_i|^2  |\nabla a_i|^2  +|a_i|^6) \phi^2\,dx.\nonumber
 \end{eqnarray}
 Using the fact that $\Delta a_i=-d^*da_i$ in $U_i$,  we have
\begin{equation}\label{xn^2}
\frac{\partial^2 a_i}{\partial x_n^2}=-\sum_{j=1}^{n-1}\frac{\partial^2 a_i}{\partial x_j^2}-D_{a_i}^*F_{a_i} +\nabla a_i\#a_i+a_i\#a_i\#a_i .
\end{equation}
Using \eqref{GD1}-\eqref{xn^2}, it implies
\begin{equation}\label{x1}
 \int_{B_1^+} \left|\frac{\partial^2 a_i}{\partial x_n^2}\right|^2 \phi^2\,dv  \leq C\int_{U_i}\left ( |\nabla \tilde{\nabla}a_i|^2 +\left| \nabla_{a_i} F_{a_i}\right|^2 +|a_i|^2|\nabla a_i|^2+|a_i|^6\right )\phi^2\,dv
\end{equation}
Using H\"older's inequality and Sobolev's inequality, we have
\begin{eqnarray}\label{x2}
&& \int_{B_1^+}|a_i|^2|\nabla a_i|^2\phi^2\,dv \\
&\leq&
\left(\int_{B_1^+}|a_i|^4\,dv\right)^{1/2}
\left(\int_{B_1^+}|\nabla a_i|^4\phi^4\,dv\right)^{1/2}
\nonumber\\
&\leq&
C\left(\int_{B_1^+}|a_i|^4\,dv\right)^{1/2}
\int_{B_1^+}
\left(|\nabla(\phi\nabla a_i)|^2+\phi^2|\nabla a_i|^2\right)\,dv
\nonumber\\
&\leq&
C\varepsilon\int_{B_1^+}
\left(\phi^2|\nabla^2a_i|^2+|\nabla a_i|^2\right)\,dv  \nonumber
\end{eqnarray}
and
\begin{eqnarray}\label{x3}
&& \int_{B_1^+}|a_i|^6\phi^2\,dv
 \leq
\left(\int_{B_1^+}|a_i|^4\,dv\right)^{1/2}
\left(\int_{B_1^+}|a_i|^8\phi^4\,dv\right)^{1/2}
\\
&\leq&
C\left(\int_{B_1^+}|a_i|^4\,dv\right)^{1/2}
\int_{B_1^+}
\left(|\nabla(|a_i|^2\phi)|^2+|a_i|^4\phi^2\right)\,dv
\nonumber\\
&\leq&
C\varepsilon\int_{B_1^+}
|a_i|^2|\nabla a_i|^2\phi^2
+|a_i|^4\bigl(|\nabla\phi|^2+\phi^2\bigr)\,dv .
\label{x3}\nonumber
\end{eqnarray}

Using \eqref{GD1}-\eqref{x3}  and choosing $\varepsilon$ sufficiently small, we obtain
\begin{eqnarray}\label{3.22}
         \int_{B_{1/2}^+} \abs{\nabla^2 a_i}^2\,dv \leq C    \int_{B_1^+}   |a_i|^2+|\nabla a_i|^2 + |\nabla_{a_i} F_{a_i} |^2\,dv.
\end{eqnarray}

Let $B_{r_0}(x_k)\subset U_i$. By a similar way to prove \eqref{3.22}, we have
\begin{eqnarray}\label{3.23}
         &&\int_{B_{r_0/2}(x_k)} \abs{\nabla^2 a_i}^2\,dv
         \\&\leq& C    \int_{B_{r_0}(x_k) }  ( r_0^{-4}|a_i|^2+r_0^{-2}|\nabla a_i|^2    +|\nabla_{a_i} F_{a_i} |^2)\,dv.\nonumber
\end{eqnarray}

   Let $\bar U_i$   be covered by a finite number of open balls $\{B_{r_0/2}(x_k)\}$.    Combining \eqref{3.23} with \eqref{3.22} yields
\begin{equation*}
\begin{aligned}
\|\nabla^2 a_i\|_{L^2(U_i)}
\le C\Big(
    &\|\nabla_{a_i} F_{a_i}\|_{L^2(U_i)}
    + r_0^{-1}\|\nabla a_i\|_{L^2(U_i)}+ r_0^{-2}\|a_i\|_{L^2(U_i)}
\Big).
\end{aligned}
\end{equation*}
This proves \eqref{3.5.1}.
\end{proof}

 As a consequence of Lemma \ref{Lemma 3.2}, we prove Theorem \ref{Theorem 4.2}.
\begin{proof}
Let $D_{A_0}$  be a connection with $A_0\in W^{1, 2}(M)$.
Assume that for  a sufficient small $\varepsilon_0>0$, there exits a constant $r_0>0$ such that
\begin{equation}  \label{S-C}
\sup_{x_0\in M}\int_{B_{r_0}(x_0)} |F_{A_0}|^{2}\,dv < \varepsilon_0.
\end{equation}
  Then, there exists a sequence of  smooth connections $\{A^k(0)\}_{k=1}^{\infty}$ such that as $k\to \infty$, $A^k(0)$ converges strongly to $A_0 \in W^{1, 2}(M)$ and $\int_{U_i} |F_{A^k(0)}|^{2}\,dv \leq \varepsilon_0$ for a sufficiently large $k$ (see Corollary 2.2. in \cite{U2}).
  Since $A^k(0)$ is smooth, there exists a unique solution $D_{A^k(t)}=D_{\text{ref}}+A^k(t)$    of the Yang-Mills flow  with initial connection $A^k(0)$  for some maximal existence time $T_k>0$.
 We claim that there exists a uniform constant $t_1>0$ such that $T_k\geq t_1$ for all $k\geq 1$.

 For sufficiently small $\varepsilon >0$,  it follows from using \eqref{S-C}  that for any point $x_0\in M$,
\[\int_{B_{2r_0}(x_0)} |F_{A^k_0}|^{2} \,dx< K\varepsilon_0\leq \frac {\varepsilon} 2.  \]
It follows from Lemma \ref{Lemma 1}  and a covering argument over $M$ that
  \begin{equation*}  \int_{B_{r_0}(x_0)} |F_{A^k(t)}|^{2}  \,dv\leq \int_{B_{2r_0}(x_0)} |F_{A^k_0}|^{2}  \,dv +\frac {Ct}{r_0^2} \sup_{0\leq s\leq t} \max_{x\in M}\int_{B_{r_0}(x)} |F_{A^k(s)}|^{2}  \,dv  \end{equation*}
  which implies
\begin{equation*} \max_{x\in M} \int_{B_{r_0}(x)} |F_{A^k(t)}|^{2}  \,dv\leq \varepsilon \end{equation*}
for all $t\in [0,\frac { r_0^2}{2C}]$. Then, the smooth solution $A^k(t)$ of the Yang-Mills flow can be extended to the uniform time $t_1 =  \frac {r_0^2}{2C}>0$  such that
  \begin{equation}  \label{LEY}\sup_{0\leq t\leq t_1}\int_{B_{r_0}(x_0)}|F_{A^k(t)}|^{2}\,dv\leq \varepsilon  \end{equation}
and
 \begin{equation*} \int_{0}^{t_1 }\int_{M}   |\nabla_{A^k}
    F_{A^k}|^2\,dv\,dt\leq C. \end{equation*}
Applying Lemma \ref{Lemma 3.2},  there exist gauge
transformations $S_i^{k}(t)$ and connections
$D_{a^k_i}=(S_i^k)^*(D_{A^k})=d+a_i^k$ in  $U_i$ such that
\[  d^*  a_i^k=0 \quad \mbox {in } U_i,\quad a_i\cdot \nu=0 \mbox { on }\partial  U_i,\]
satisfying
 \begin{equation}  \label{ak}\int_{U_i} r_0^{-2}|a_i^k(t)|^{2}+|\nabla a_i^k(t)|^{2}\,dv\leq C \int_{U_i} |F_{a_i^k(t)}|^{2}\,dv\leq C\varepsilon \end{equation}
for  all $t\in [0, t_1]$,
and for which $D_{a_i^k}$ is a smooth solution of the equation
\begin{eqnarray}\label{4.5}
    \pfrac{a_i^k}{t}=-D_{a_i^k}^*F_{a_i^k}+D_{a_i^k} s_i^k
    \end{eqnarray}
   in $U_i\times [0,t_1]$, where \[s_i^k(t)=(S_i^k)^{-1}(t)\circ \frac d{dt} S_i^k(t), \quad \int_{U_i}s_i^k(x,t)\,dx=0.   \] and
  \begin{eqnarray}\label{3.26}
  &&\int_{0}^{t_1 }\int_{U_i}   \left (|\pfrac{a_i^k}{t}|^2+|\nabla^2 a^k_i|^2 +|\nabla s_i^k|^2\right )\,dv\,dt\\
  &\leq & \int_{0}^{t_1 }\int_{U_i}  ( |\nabla_{a^k}
    F_{a_i^k}|^2+\frac{1}{r_0^2}|F_{a_i^k}|^2)\,dv\,dt\leq C(1+\frac{t_1}{r_0^2})\int_M |F_{A^k_0}|^2\,dv.\nonumber
 \end{eqnarray}
For any $\delta>0$ and each $l \geq 1$, it follows from using Lemma \ref{Lemma 6} that
\begin{equation*}
\begin{aligned}
&\sup_{\delta\leq t\leq t_1}
  \int_{U_i}\sum_{m=1}^{\ell}
  \left|\nabla_{a_i^k}^{m}F_{a_i^k}\right|^2\,dv
 + \int_{\delta}^{t_1}\int_{U_i}\sum_{m=1}^{\ell}
  \left|\nabla_{a_i^k}^{m}F_{a_i^k}\right|^2\,dv\,dt
 \leq C(\delta,\ell),
\end{aligned}
\end{equation*}
 for uniform  constants $C(\delta, l)>0$. Since $d^* a_i^k=0$ in $U_i$, we have
\[\Delta a_i^k= -d^*d a^k= -d^*F_{a^k}-d^*(a^k\wedge a^k).\]
Then, using the $L^p$-estimates, the Sobolev inequality with  \eqref{ak}, a  standard induction argument (see, e.g., \cite{CCH}) yields
\begin{equation*}
\begin{aligned}
\sup_{\delta\leq t\leq t_1}
  \int_{(1-\theta) U_i}
  \left(
    |\nabla^{l+1}a_i^k|^2
    +|\nabla_{a_i^k}^{l}F_{a_i^k}|^2
  \right)\,dv
 \leq C(\delta,\theta, l)
\end{aligned}
\end{equation*} for each $l\geq 1$ and a small $\theta >0$.
Using
 \begin{equation}\label{as}
d^*d s_i^k=a\#D_{a_i^k}^*F_{a}+a\#\nabla s_i^k ,
       \end{equation}
we have
\begin{equation}
\begin{aligned}
  \sup_{\delta\leq t\leq t_1}\int_{ (1-\theta) U_i}
  |\nabla^2 s_i^k|^{l+1}\,dv\,dt \leq
  C(\delta, l, \theta ).
\end{aligned}
\end{equation}
 Then, as $k\to\infty$, $a^{k}_i$ smoothly locally converges to $a_i$ in $U_i$ for any $t\in [\delta, {t_1}]$ satisfying that $d^*a_i(t)=0$ in $U_i$ with $a_i(t)\cdot\nu =0$ on $\partial U_i$  such that for  $t\in [\delta , {t_1}]$,
 \[\int_{U_i} \frac 1{r_0^{2}}|a_i(t)|^{2}+|\nabla a_i(t)|^{2}\,dv\leq C \int_{U_i} |F_{a_i
 (t)}|^{2}\,dv,\]
 \begin{eqnarray}\label{3.27}
  &&\int_{\delta}^{t_1 }\int_{U_i}   \left (|\pfrac{a_i}{t}|^2+|\nabla^2 a_i|^2 +|\nabla s_i|^2\right )\,dv\,dt\\
  &\leq & \int_{\delta}^{t_1 }\int_{U_i}  ( |\nabla_{a_i}
    F_{a_i}|^2+\frac{1}{r_0^2}|F_{a_i}|^2)\,dv\,dt\leq C(1+\frac{t_1}{r_0^2})\int_M |F_{A_0}|^2\,dv.\nonumber
 \end{eqnarray}
Letting $\delta \to 0$,  we obtain
  \begin{eqnarray}\label{3.27}
  &&\int_{0}^{t_1 }\int_{U_i}   \left (|\pfrac{a_i}{t}|^2+|\nabla^2 a_i|^2 +|\nabla s_i|^2\right )\,dv\,dt\\
   &\leq & C\int_{0}^{t_1 }\int_{U_i}  ( |\nabla_{a_i}
    F_{a_i}|^2+\frac{1}{r_0^2}|F_{a_i}|^2)\,dv\,dt\leq C(1+\frac{t_1}{R_0^2})\int_M |F_{A_0}|^2\,dv.\nonumber
       \end{eqnarray}
       Moreover, $a_i$ is a  solution to
  \begin{eqnarray}\label{4.8}
 &&\frac d{dt} a_i= - D_{a_i}^*F_{a_i}+ D_{a_i} s_i\quad \mbox{in $U_i\times (0,t_1]$}
    \end{eqnarray} and $d^*a=0$ in $U$,
 where \[ \int_{U_i}s_i(x,t)\,dv=0.   \]

Using Lemma \ref{Lemma 1}, we have
\begin{equation}\int_0^{t_1}\int_M \abs{\pfrac{A^k}{t}}^2 dv\,dt\leq  C\left(1+ \frac{t_1}{r_0^2}\right)\int_M  |F_{A^k_0}|^2
  \,dv \end{equation}
and
 \begin{eqnarray}\label{3.28}
\frac {d}{d t} \int_M  |A^k(t) |^2 dv&=&2   \int_M   \left <\frac{d }{d t}A^k(t), A^k(t)\right >
  \,dv\\
  &\leq &  \int_M \abs{\pfrac{A^k}{t}}^2 dv + \int_M  |A^k(t) |^2 dv. \nonumber \end{eqnarray}
Using Gronwall's inequality and \eqref{3.28}, we have
\begin{equation*}
\int_M |A^k(t)|^2\,dv
\leq e^{t_1}\left[
\int_M |A^k(0)|^2\,dv
+C\left(1+\frac{t_1}{r_0^2}\right)
\int_M |F_{A^k_0}|^2\,dv
\right].
\end{equation*}
   In $U_i$, we note
\begin{equation*}
 a^k_{i}(t)= (S_i^k)^{-1}(t)  d S^k_{i}(t) + (S_i^k)^{-1}(t) A^k(t)S^k_{i}(t) .\end{equation*}
Up to a subsequence,  $S^k_{i}$ weakly converges to $S_i$ in $L^{\infty}([0, t_1]; H^{1,2}(U_i))$ and $W^{1,2}([0, t_1]; L^2(U_i))$ as $k\to \infty$.
Furthermore, $A^k$ weakly converges to $A$ in $L^{\infty}(0, t_1; L^{2}(U_i))$ and $W^{1,2}([0, t_1]; L^2(U_i))$, where $A$ is a weak solution of the Yang-Mills flow with initial value $A_0\in W^{1,2}(M)$.
This shows that $A(t)$ is a weak solution of  the Yang-Mills flow (\ref{YMF}) satisfying

\[A\in C^0\left ( [0, t_1]; L^2( \Omega^1(Ad \,E ))\right )\cap W^{1,2}\left ([0, t_1]; L^2(\Omega^1(Ad \,E ))\right )\] and
		 \[ F_A\in C^0\left ( [0, t_1]; L^2(\Omega^2(Ad \,E ))\right ). \]
Fix $\tilde t>0$. Then $S_i^*(\tilde t)(A^k(x,t))$  converges to $\tilde A(t)=S_i^*(\tilde t)(A(x,t))$ in $C^{\infty}_{loc} (U_i\times (0, t_1])$, which is   a weak solution of the Yang-Mills flow with initial data $S_i^*(\tilde t)(A_0)$  in each   $U_i$ and smooth locally in $U_i$ for $t>0$.
\end{proof}

\section{Uniqueness of  weak solutions in dimension four}

In this section, we  prove  uniqueness of weak solutions of the Yang-Mills flow. For the compact manifold $M$,   there are a finite number of points $y_i\in M$, $i=1,\cdots,L$,  such that $\{U_i=B_{r_0}(y_i)\}_{i=1}^L$ covers $M$.  Moreover, for each $x\in M$ there are at most finite number $K$  of nonempty intersections of  different  balls $U_j$. For the open cover $\{U_i\}$, we have the local trivialization
$\pi_i:E|_{U_i}\to U_i\times\mathbb R^l$. In $U_i\cap U_j\neq \emptyset$, the transition maps $T_{ij}:U_i\cap U_j\to G$ are defined by
$$
\pi_i\circ\pi_j^{-1}(x,v)=(x,T_{ij}(x)v).
$$
The transition maps satisfies
$$
T_{ii}=I,\qquad T_{ji}=T_{ij}^{-1},\qquad
T_{ij}T_{jk}=T_{ik}.
$$ If we write $D_{A}=d+A_i$ in local trivialization of $U_i$.  In $U_i\cap U_j\neq \emptyset$, we have $A_j=T_{ij}^{-1}dT_{ij}+T_{ij}^{-1}A_iT_{ij}$.

\begin{lem}  \label{Lemma 4.1} Let $A(t)$ be a weak solution to the Yang-Mills flow (\ref{YMF}) in Theorem \ref{Theorem 1.2} obtained by Struwe in \cite{St2}.
 Then,
   for each  ball $U_i$, there exist    a   gauge
transformation  $S_i(t)\in W^{1,2}(U_i)$    and  a  connection
$D_{a_i}=S_i^*(D_{A})=d+a_i$ such that
 \begin{equation}\label{UG}  d^*  a_i(t)=0 \quad \mbox {in } U_i,\quad a_i(t)\cdot \nu=0 \mbox { on }\partial  U_i \end{equation} for each $t\in [0, t_1]$ and $D_{a_i}$  is a weak solution of the equation
  \begin{equation}\label{6.2}
    \pfrac{a_i}{t}=-D_{a_i}^*F_{a_i}+D_{a_i} s_i\quad \mbox{in $ U_i\times (0,t_1]$,}
     \end{equation}
     where \[s_i(t)=S_i^{-1}(t)\circ \frac d{dt} S_i(t)\in W^{1,2}(U_i), \quad \int_{U_i}s_i(x,t)\,dv=0.  \]
  Moreover, we have
  \begin{equation}\label{6.4}
    \sup_{0\leq t\leq t_1}
    \int_{U_i}r_0^{-2} \abs{a_i(t)}^2 +\abs{\nabla  a_i(t)}^2 dv \leq C\varepsilon.
    \end{equation}
  and
  \begin{equation}\label{6.5.1}
  \begin{aligned}
 \int_0^{t_1}\int_{U_i}\left|\frac{\partial a_i}{\partial t}\right|^2+ \abs{\nabla^2a_i }^2 +|\nabla s_i|^2dvdt   \leq C\left(1+\frac {t_1}{r_0^2}\right)\int_M|F_{A_0}|^2\, dv.
  \end{aligned}
  \end{equation}
\end{lem}
\begin{proof}Let $A$ be a weak solution to Theorem  \ref{Theorem 1.2} obtained by Struwe in \cite{St2} such that for any $t>0$, $A(t)$ is gauge to a  weak solution  $\hat A (t)$, which is smooth for $0<t<T$. Moreover, there exist a sequence $t_k \to 0$ and  a sequence of smooth gauge transformations
and  such that as $k\to \infty$,   $S_k\to S_0$ in $W^{1,2}(M)$, $S^*_k(D_{\hat A(t_k)})\to S_0^*(D_{A_0})$ in $W^{1,2}(M)$ and $A(t)= S^*_0(\hat A(t))$ is a weak solution of the Yang-Mills flow.

By Lemma \ref{Lemma 1}, there exists a constant  $t_1>0$ such that
\[\sup_{0\leq t\leq t_1} \max_{x_0\in M} \int_{B_{r_0}(x_0)} |F_{\hat A}(\cdot ,t)|^{2}dv \leq \varepsilon. \]
Applying  Uhlenbeck's gauge fixing theorem,  there exist gauge
transformations $\tilde S_{i; k} (t)$ and connections
$D_{\tilde a_{i; k} }=(\tilde S_{i; k} )^*(D_{A})=d+\tilde a_{i; k} $ in  $U_i$ such that for all $t\in [t_k, t_1]$,
\[  d^* \tilde a_{i; k}  (t) =0 \quad \mbox {in } U_i,\quad \tilde a_{i; k}  (t)\cdot \nu=0 \mbox { on }\partial  U_i \]
satisfying
 \[\int_{U_i} r_0^{-2}|\tilde a_{i; k} (t)|^{2}+|\nabla \tilde a_{i; k} (t)|^{2}\,dv\leq C \int_{U_i} |F_{\tilde a_{i; k}  (t)}|^{2}\,dv\leq C\varepsilon\]
for  all $t\in [t_k, t_1]$  and
 \begin{equation*}
  \begin{aligned}
 \int_{t_k}^{t_1}\int_{U_i}\left|\frac{\partial  \tilde a_{i; k} }{\partial t}\right|^2+ \abs{\nabla^2 \tilde a_{i; k} }^2 +|\nabla \tilde s_i|^2dvdt   \leq C\left(1+\frac {t_1}{r_0^2}\right)\int_M|F_{A_0}|^2\, dv.
  \end{aligned}
  \end{equation*}and $D_{a_{i; k} }$  is a smooth solution of the equation
  \begin{equation}\label{6.2.1}
    \pfrac{\tilde a_{i; k}}{t}=-D_{\tilde a_{i: k}}^*F_{\tilde a_{i; k} }+D_{\tilde a_{i; k} } \tilde s_{i; k}\quad \mbox{in $ U_i\times [t_k,t_1]$,}
     \end{equation}
     where $\tilde s_{i; k} (t)=\tilde S_{i; k} ^{-1}(t)\circ \frac d{dt} \tilde S_{i; k} (t)$ and  $\int_{U_i}s_{i;k}(x,t)\,dv=0$ for all $t\in [t_k, t_1]$.

     Let $\delta >0$ be a small constant.  Since $\tilde a_{i; k} (t)$ for each $t\in [\delta, t_1]$ is smooth, it follows from the same argument in Theorem \ref{Theorem 4.2} that as $k\to \infty$, $\tilde a_{i; k}(t)$  converges to $a_i(t)$ in $U_i\times [\delta, t_1]$, in which $a_i(t)$ is a solution to
      \begin{equation}\label{6.3.1}
    \pfrac{ a_i}{t}=-D_{a_i}^*F_{a_i}+D_{  a_i}   s_i\quad \mbox{in $U_i\times [\delta ,t_1]$,}
     \end{equation}
     where $ s_i(t)= S_i^{-1}(t)\circ \frac d{dt}  S_i(t)$ in $U_i$ and $\int_{U_i}s_{i}(x,t)\,dv=0$. Moreover, letting $\delta \to 0$, $D_{a_{i}}$ satisfies \eqref{6.2}-\eqref{6.5.1} for all $t\in [0,  t_1]$.
\end{proof}

We note that for this type of weak solution, the estimate in Lemma \ref{Lemma 2} holds as well.
We prove a time weighted estimate for weak solutions of the Yang-Mills flow.
\begin{lem}  \label{Lemma 4.2} Let $A(t)$ be a weak solution to the Yang-Mills flow (\ref{YMF})  in Theorem \ref{Theorem 4.2} or Lemma \ref{Lemma 4.1}. Assume that there are constants  $\varepsilon>0$ and $0 < r_0$ such that
 \[\sup_{0\leq t\leq T} \max_{x_0\in M} \int_{B_{r_0}(x_0)} |F_A(\cdot ,t)|^{2}dv \leq \varepsilon. \]
 Then there exists a constant $C > 0$ such that
\begin{equation}\label{4.1}
\begin{aligned}
&\sup_{0<t\leq T}
t\int_M|\nabla_A F_A(t)|^2\,dv
+\int_0^Tt\int_M|\nabla_A^2F_A(t)|^2\,dv\,dt\\
&\qquad\leq
C\left[1+(1+r_0^{-2})T\right]^2
\int_M|F_{A_0}|^2\,dv.
\end{aligned}
\end{equation}
\end{lem}

\begin{proof}
Let $A(t)$ be a weak solution of the Yang-Mills flow in $M\times[0,T]$, which is a weak limit of smooth solutions. From Lemma \ref{Lemma 2}, we have $\int_M |\nabla_A F_A|^2 dv \in L^1 (0,T)$. We claim that there exists a sequence $t_m\rightarrow 0$ such that
\begin{equation}
t_m \int_M |\nabla_A F_A|^2(t_m)dv\longrightarrow0.
\end{equation}
Otherwise, there exist constants $\eta>0$ and $t^*$ such that
\begin{equation}
    t\int_M |\nabla_A F_A|^2(t)dv\geq \eta,
\end{equation}
for all $t\in (0,t^*)$. This implies that
\begin{equation}
 \int^{t_*}_0 \int_M |\nabla_A F_A|^2(t)dvdt\geq \eta  \int^{t^*}_0 \frac{1}{t}dt=\infty.
\end{equation}
This is a contradiction. By Theorem \ref{Theorem 4.2}, $A(t)$ is gauge equivalent to a smooth solution of the Yang-Mills flow $\tilde{A}(t)$ in $M\times [t_m,T]$.

Using the Bianchi identity  $D_{\tilde{A}} F_{\tilde{A}}=0$ and the Weitzenb\"ock formula, we have
\begin{equation}\label{4.2}
\frac{\partial F_{\tilde{A}}}{\partial t}
=-D_{\tilde{A}}D_{\tilde{A}}^*F_{\tilde{A}}
=-\nabla_{\tilde{A}}^*\nabla_{\tilde{A}}F_{\tilde{A}}
+F_{\tilde{A}}\#F_{\tilde{A}}
+R_M\#F_{\tilde{A}}.
\end{equation}
Using the well-known Ricci identity, we obtain
\begin{equation}\label{RIk}
\frac{\partial}{\partial t}\nabla_{\tilde{A}}F_{\tilde{A}}
=-\nabla_{\tilde{A}}^*\nabla_{\tilde{A}}^2F_{\tilde{A}}
+(Rm+F_{\tilde{A}})\#\nabla_{\tilde{A}}F_{\tilde{A}}
+\nabla Rm\#F_{\tilde{A}}.
\end{equation}
 Multiplying \eqref{RIk} by $\nabla_{\tilde{A}}F_{\tilde{A}}$
and integrating over $M$ yields
\begin{equation}
\begin{aligned}
&\quad \frac12\frac{d}{dt}
\int_M |\nabla_{\tilde{A}}F_{\tilde{A}}|^2\,dv
+\int_M |\nabla_{\tilde{A}}^2F_{\tilde{A}}|^2\,dv\\
&\leq
C\int_M (1+|F_{\tilde{A}}|)
|\nabla_{\tilde{A}}F_{\tilde{A}}|^2\,dv
+C\int_M |F_{A_0}|^2\,dv\\
&\leq
C\sum_{i=1}^L
\left(\int_{U_i}|F_{\tilde{A}}|^2\,dv\right)^{1/2}
\left(\int_{U_i}
|\nabla_{\tilde{A}}F_{\tilde{A}}|^4\,dv\right)^{1/2}\\
&
+C\int_M|\nabla_{\tilde{A}}F_{\tilde{A}}|^2\,dv
+C\int_M|F_{A_0}|^2\,dv\\
&\quad \leq
C\varepsilon^{1/2}
\int_M
\left(
|\nabla_{\tilde{A}}^2F_{\tilde{A}}|^2
+\frac{1}{r_0^2}
|\nabla_{\tilde{A}}F_{\tilde{A}}|^2
\right)\,dv
+C\int_M|\nabla_{\tilde{A}}F_{\tilde{A}}|^2\,dv\\
&\quad
+C\int_M|F_{A_0}|^2\,dv.
\end{aligned}
\end{equation}
Choosing $\varepsilon$ to be sufficiently small, we obtain
\begin{equation}\label{4.8'}
\begin{aligned}
&\frac{d}{dt}
\int_M|\nabla_{\tilde{A}}F_{\tilde{A}}|^2\,dv
+\int_M|\nabla_{\tilde{A}}^2F_{\tilde{A}}|^2\,dv\\
&\leq
C\left(1+\frac{1}{r_0^2}\right)
\int_M|\nabla_{\tilde{A}}F_{\tilde{A}}|^2\,dv
+C\int_M|F_{A_0}|^2\,dv.
\end{aligned}
\end{equation}
Multiplying \eqref{4.8'} by time, we get
\begin{equation}
\begin{aligned}
&\quad \frac{d}{d\tau}
\left(\tau\int_M
|\nabla_{\tilde{A}}F_{\tilde{A}}|^2\,dv\right)
+\tau\int_M
|\nabla_{\tilde{A}}^2F_{\tilde{A}}|^2\,dv\\
&\leq
\left(1+C\tau(1+r_0^{-2})\right)
\int_M|\nabla_{\tilde{A}}F_{\tilde{A}}|^2\,dv
+c\tau\int_M|F_{A_0}|^2\,dv.
\end{aligned}
\end{equation}
Integrating over $[t_m,t]$ yields
\begin{equation}
\begin{aligned}
&t\int_M
\left|\nabla_{\tilde{A}(t)}F_{\tilde{A}(t)}\right|^2\,dv
+\int_{t_m}^t\tau\int_M
\left|\nabla_{\tilde{A}(\tau)}^2
F_{\tilde{A}(\tau)}\right|^2\,dv\,d\tau\\
&\leq\quad
\left[1+C(1+r_0^{-2})t\right]
\int_{t_m}^t\int_M
\left|\nabla_{\tilde{A}(\tau)}
F_{\tilde{A}(\tau)}\right|^2\,dv\,d\tau\\
&\qquad +
t_m\int_M
\left|\nabla_{\tilde{A}(t_m)}
F_{\tilde{A}(t_m)}\right|^2\,dv++
\frac{C}{2}(t^2-t_m^2)
\int_M|F_{A_0}|^2\,dv.
\end{aligned}
\end{equation}
Using gauge invariance and letting $t_m \rightarrow 0$, we obtain
\begin{equation}
\begin{aligned}
&t\int_M
\left|\nabla_{A(t)}F_{A(t)}\right|^2dv
+\int_0^t\tau\int_M
\left|\nabla_{A(\tau)}^2F_{A(\tau)}\right|^2dvd\tau\\
&\qquad\leq
\left[1+C(1+r_0^{-2})t\right]
\int_0^t\int_M
\left|\nabla_{A(\tau)}F_{A(\tau)}\right|^2,dvd\tau
+Ct^2\int_M|F_{A_0}|^2dv.
\end{aligned}
\end{equation}
Combining this with Lemma \ref{Lemma 2}, we proves the claim.
\end{proof}

Now we give a proof of Theorem  \ref{Theorem 1.2}.
\begin{proof}

For $k=1,2$, let $D_{A^k}$     be two weak solutions of the Yang-Mills flow obtained in Theorem \ref{Theorem 4.2},  with the same initial value $A_0 \in W^{1, 2}(M)$  satisfying \eqref{IU}
such that
 \[\sup_{0\leq t\leq t_1} \max_{x_0\in M} \int_{B_{r_0}(x_0)} |F_{A^k}(\cdot ,t)|^{2}dv \leq \varepsilon. \]
By the  existence in  Theorem  \ref{Theorem 4.2} and Lemma \ref{Lemma 4.1}, we have the following properties:

For $k=1,2$ and each ball $U_i$ with $1\leq i\leq J$, there are two gauge
transformations $S_i^{k}(t)$   and     connections
$D_{a^k}=(S_i^k)^*(D_{A^k})=d+a_i^k(t)$ such that
\begin{equation}\label{4.1}
 d^*  a_i^k(t)=0 \quad \mbox {in } U_i, \quad a_i(t)\cdot \nu=0 \mbox { on }\partial  U_i
 \end{equation}
for all $t\in [0, t_1]$,
and  $D_{a_i^k}$ are    solutions of the equation
\begin{eqnarray}\label{4.2}
    \pfrac{a_i^k}{t}=-D_{a_i^k}^*F_{a_i^k}+D_{a_i^k} s_i^k
    \end{eqnarray}
   in $U_i\times [0,t_1]$, where \[s_i^k(t)=(S_i^k)^{-1}(t)\circ \frac d{dt} S_i^k(t),\quad \int_{U_i}s_i^k(t) \,dv=0.  \]
   For all $i=1,\cdots,L$, we have
  \begin{equation}\label{4.3}
    \int_{U_i}r_0^{-2} \abs{a^k_i(t)}^2 +\abs{\nabla  a^k_i(t)}^2 dv\leq C\int_{U_i}|F_{a^k_i(t)}|^2 dv\leq C\varepsilon
     \end{equation}
  and
  \begin{align}\label{4.4}
     \int_0^{t_1} \int_{U_i} \abs{\nabla^2a_i^k }^2 + | \pfrac{a_i^k}{t}|^2+|\nabla s_i^k|^2dvdt
    \leq   C \int_0^{t_1}\int_{U_i} |\nabla_{a_i^k}F_{a_i^k}|^2+|F_{a_i^k}|^2\,dv\,dt.
  \end{align}
In $U_i$, $a^1_i(0)$ and $a^2_i(0)$ may be different, but there exists a gauge transformation such that
\[d q_i= q_i a^2_i(0)-a^1_i(0)q_i . \]
Then, using \eqref{4.1}, we have
\begin{equation}
\begin{cases}
-\Delta q_i
= g^{jk}\left(
a^1_{i, j}(0)\frac{\partial q_i}{\partial x_k}
-\frac{\partial q_i}{\partial x_j}a^2_{i, k}(0)
\right)
 & \text{in } U_i, \\
\displaystyle
\frac{\partial q_i }{\partial \nu}=0
& \text{on } \partial U_i.
\end{cases}
\end{equation}
  Applying the elliptic estimate, Poincare's inequality and  Sobolev's inequality, we obtain
\begin{equation}
\begin{aligned}
&\|q_i-(q_i)_{(U_i)}\|_{W^{2,2}(U_i)} \le C(\| a^1_i(0)\|_{L^4(U_i)})+\| a^2_i(0)\|_{L^4(U_i)})\|\nabla q_i\|_{L^4(U_i)}\\
&\le C(\|a_i^1(0)\|_4+\|a_i^2(0)\|_4)\|q_i-(q_i)_{U_i}\|_{W^{2,2}(U_i)}.
\end{aligned}
\end{equation}
Using \eqref{4.3}, it implies that $q_i$ is constant gauge transformation.

From now on, without loss of generality,  we can assume that $a^1_i(0)=a^2_i(0)$ and $S^1_i(0)=S^2_i(0)$.

By Lemma \ref{Lemma 4.2}, we have
\begin{equation}\label{4.15'}
\begin{aligned}
&t\int_M
\left|\nabla_{A^k(t)}F_{A^k(t)}\right|^2\,dv+
\int_0^t\tau\int_M
\left|\nabla_{A^k(\tau)}^2F_{A^k(\tau)}\right|^2\,dv\,d\tau\\
&\qquad\leq
\left[1+C(1+r_0^{-2})t\right]
\int_0^t\int_M
\left|\nabla_{A^k(\tau)}F_{A^k(\tau)}\right|^2\,dv\,d\tau+
Ct^2\int_M|F_{A_0}|^2\,dv.
\end{aligned}
\end{equation}
Using Lemma \ref{Lemma 2}, the Lebesgue continuity implies
\begin{equation}\label{4.16'}
t\sum_{k=1}^2\int_M
\left|\nabla_{A^k(t)}F_{A^k(t)}\right|^2\,dv \rightarrow0, \quad
\qquad\text{as }t\rightarrow0.\\
\end{equation}
 We set $b_i=a_i^2-a_i^1$ and $\sigma_i =s_i^2-s_i^1$ in $U_i$. The main idea of the proof is toobtain interior estimates using integration by parts and then convert boundary estimates to interior estimates using the transition behaviors. We choose a collection of cut-off functions $\phi_i\in C^\infty_0(U_i)$ with $0\leq  \phi_i\leq 1$ and $|\nabla \phi_i| \leq \frac{C}{r_0}$, such that every point in the transition region with $0\leq \phi_i<1$ has a neighborhood in which some $\phi_j$ is identically $1$. For simiplicity, we denote
\begin{equation}
\begin{aligned}
\mathcal E(t)
&:=\sum_i\int_{U_i}
|F_{a_i^2}-F_{a_i^1}|^2\,dv,\\
\mathcal E_\phi(t)
&:=\sum_i\int_{U_i}
|F_{a_i^2}-F_{a_i^1}|^2\phi_i^2\,dv,\\
\mathcal D(t)
&:=\sum_i\int_{U_i}
|\nabla_{a_i^1}(F_{a_i^2}-F_{a_i^1})|^2\,dv,\\
\mathcal D_\phi(t)
&:=\sum_i\int_{U_i}
|\nabla_{a_i^1}(F_{a_i^2}-F_{a_i^1})|^2
\phi_i^2\,dv,\\
\mathcal P(t)
&:=\sum_i\int_{U_i}|d\sigma_i|^2\,dv.
\end{aligned}
\end{equation}
Using \eqref{4.15'}, \eqref{4.16'}, Lemma \ref{Lemma 4.2} and Sobelev's inequality, we have
\begin{equation}\label{4.11}
\begin{aligned}
K(t)&:=\sum_{i,k}\|F_{a_i^k}(t)\|_{L^4(U_i)}^2 \in L^1(0,T), \quad tK(t)\rightarrow 0\\
\Lambda(t)&=\sum_{i,k}
\|\nabla_{a_i^k}F_{a_i^k}(t)\|_{L^4(U_i)}^2,\quad  \int^T_0 t \Lambda(t) dt< \infty.
\end{aligned}
\end{equation}
 We have \begin{equation}\label{4.12}
\begin{aligned}
db_i
&=F_{a_i^2}-F_{a_i^1}
-a_i^1\wedge b_i-b_i\wedge a_i^2, \quad d^{*}b_i=0, \text{ in } U_i\\
&\qquad b_i\cdot\nu=0 \text{ on } \partial U_i.
\end{aligned}
\end{equation}
Using Lemma \ref{L5} and the the Sobelev's inequality, we get
\begin{equation}\label{4.13}
r_0^{-1}\|b_i\|_{L^2(U_i)}
+\|\nabla b_i\|_{L^2(U_i)}+\|b_i\|_{{L^4(U_i)}}
\le C\|F_{a_i^2}-F_{a_i^1}\|_{L^2(U_i)}.
\end{equation}
Since $\sigma_i$ has average zero in $U_i$, it follows from   Sobelev's inequality and Poincare's inequality   that
\begin{equation}
\|\sigma_i\|_{L^4(U_i)}
\le C\|d\sigma_i\|_{L^2(U_i)}.
\end{equation}
We have
\begin{equation}\label{4.14}
\begin{aligned}
\frac{\partial b_i}{\partial t}-d\sigma_i
={}&-\left(
D_{a_i^2}^{*}F_{a_i^2}
-D_{a_i^1}^{*}F_{a_i^1}
\right)+[a_i^1,\sigma_i]+[b_i,s_i^2].
\end{aligned}
\end{equation}
Using \eqref{4.12}, we get
\begin{equation}
\int_{U_i}
\langle \frac{\partial b_i}{\partial t},d\sigma_i\rangle\,dv=0.
\end{equation}
This implies that
\begin{equation}
\begin{aligned}
&\quad ||\frac{\partial b_i}{\partial t}||^2_{L^2(U_i)}+||d\sigma_i||^2_{L^2(U_i)} =||\frac{\partial b_i}{\partial t}-d\sigma_i||^2_{L^2(U_i)}\\
&\leq C||D_{a_i^2}^{*}F_{a_i^2}
-D_{a_i^1}^{*}F_{a_i^1}||^2_{L^2(U_i)}+C||a^1_i||^2_{L^4(U_i)}||d\sigma_i||^2_{L^2(U_i)}\\
&\quad +C||b_i||^2_{L^4(U_i)}||s^2_i||^2_{L^4(U_i)}.
\end{aligned}
\end{equation}
Taking $\varepsilon$ sufficient small and using \eqref{4.13}, we obtain
\begin{equation}\label{4.23.1}
\begin{aligned}
&\quad ||\frac{\partial b_i}{\partial t}||^2_{L^2(U_i)}+||d\sigma_i||^2_{L^2(U_i)}\\
&\leq C||\nabla_{a^1_i}(F_{a^2_i}-F_{a^1_i})||^2_{L^2(U_i)}+C(||F_{a^2_i}||^2_{L^4(U_i)}+||s^2_i||^2_{L^4(U_i)})||F_{a^2_i}-F_{a^1_i}||^2_{L^2(U_i)}.
\end{aligned}
\end{equation}
Summing over $i$, we get
\begin{equation}\label{4.18}
\sum_i\|\frac{\partial b_i}{\partial t}\|_{L^2(U_i)}^2+\mathcal P(t)
\le C\mathcal D(t)+Ch(t)\mathcal E(t),
\end{equation}
where $
h(t):={}1+r_0^{-2}+K(t)
+\sum_{i,k}\left(
\|\nabla_{a_i^k}F_{a_i^k}(t)\|_2^2
+\|s_i^k(t)\|_4^2
\right)$.     Noting \eqref{4.4} and \eqref{4.11}, we have $h\in L^1[0,T]$.

Using the Bianchi idenity and the Weitzenbock formula, we obtain
\begin{equation}
\begin{aligned}
&\partial_t(F_{a_i^2}-F_{a_i^1})
+[s_i^1,F_{a_i^2}-F_{a_i^1}]
+[\sigma_i,F_{a_i^2}]\\
={}&-\nabla_{a_i^2}^{*}\nabla_{a_i^2}F_{a_i^2}
+\nabla_{a_i^1}^{*}\nabla_{a_i^1}F_{a_i^1}\\
&+F_{a_i^2}\#F_{a_i^2}
-F_{a_i^1}\#F_{a_i^1}
+\operatorname{Rm}\#(F_{a_i^2}-F_{a_i^1}).
\end{aligned}
\end{equation}
Multiplying by $(F_{a^2_i}-F_{a^1_i})\phi_i^2$, integrating over $U_i$ and noting that $\<[B,A],A\>=0$, we get
\begin{equation}\label{4.20}
\begin{aligned}
&\frac12\frac{d}{dt}
\int_{U_i}|F_{a_i^2}-F_{a_i^1}|^2\phi_i^2\,dv+\int_{U_i}
|\nabla_{a_i^1}(F_{a_i^2}-F_{a_i^1})|^2
\phi_i^2\,dv\\
\le{}&
C(1+r_0^{-2})\|F_{a_i^2}-F_{a_i^1}\|_2^2+C\int_{U_i}|b_i|^2|F_{a_i^2}|^2\phi_i^2\,dv\\
&+C\int_{U_i}
|b_i||\nabla_{a_i^2}F_{a_i^2}|
|F_{a_i^2}-F_{a_i^1}|\phi_i^2\,dv+C\int_{U_i}
(|F_{a_i^1}|+|F_{a_i^2}|)
|F_{a_i^2}-F_{a_i^1}|^2\phi_i^2\,dv\\
&+C\int_{U_i}
|\sigma_i||F_{a_i^2}|
|F_{a_i^2}-F_{a_i^1}|\phi_i^2\,dv.
\end{aligned}
\end{equation}
Using   Sobelev's inequality and Kato's inequality, we have
\begin{equation}
\begin{aligned}
\|\phi_i(F_{a_i^2}-F_{a_i^1})\|_4
\le C\Big(
&\|\phi_i\nabla_{a_i^1}(F_{a_i^2}-F_{a_i^1})\|_2+r_0^{-1}\|F_{a_i^2}-F_{a_i^1}\|_2
\Big).
\end{aligned}
\end{equation}
We can then estimate the right hand side   of \eqref{4.20} by
\begin{equation}
\begin{aligned}
&\int_{U_i}
|\sigma_i||F_{a_i^2}|
|F_{a_i^2}-F_{a_i^1}|\phi_i^2\,dv\\
&\le
\|\sigma_i\|_{L^4(U_i)}\|F_{a_i^2}\|_{L^2(U_i)}
\|\phi_i(F_{a_i^2}-F_{a_i^1})\|_{L^4(U_i)}\\
&\le
\eta\|\phi_i\nabla_{a_i^1}(F_{a_i^2}-F_{a_i^1})\|_{L^2(U_i)}^2
+C_\eta\varepsilon\|d\sigma_i\|_{L^2(U_i)}^2
+Cr_0^{-2}\|F_{a_i^2}-F_{a_i^1}\|_{L^2(U_i)}^2
\end{aligned}
\end{equation}
and
\begin{equation}\label{4.23}
\begin{aligned}
&\int_{U_i}
|b_i||\nabla_{a_i^2}F_{a_i^2}|
|F_{a_i^2}-F_{a_i^1}|\phi_i^2\,dv\\
&\qquad\le
\eta\|\phi_i\nabla_{a_i^1}(F_{a_i^2}-F_{a_i^1})\|_{L^2(U_i)}^2
\\ &+C_\eta\left(
r_0^{-2}
+\|\nabla_{a_i^2}F_{a_i^2}\|_{L^2(U_i)}^2
\right)
\|F_{a_i^2}-F_{a_i^1}\|_{L^2(U_i)}^2
\end{aligned}
\end{equation}
for all $\eta > 0$. Combining \eqref{4.20}-\eqref{4.23} and summing over $i$, we obtain
\begin{equation}\label{4.24}
\begin{aligned}
\mathcal E_\phi(t)
+\int_0^t\mathcal D_\phi(\tau)\,d\tau
\le{}&
C\int_0^t h(\tau)\mathcal E(\tau)\,d\tau+C\varepsilon
\int_0^t\mathcal P(\tau)\,d\tau.
\end{aligned}
\end{equation}
This is the desired interior estimate. Now, we   estimate the transition terms from $U_i$ to $U_{j}$, which  allows us to get a full estimate on each $U_i$. In the  intersection $U_i\cap U_{j}$, there exist a gauge transformation $S_{ij; k}(t):=(S^k_i)^{-1}T_{ij}S^k_j$  (see Lemma 3.5 of \cite{U2} or \cite{Se})   such that
\begin{eqnarray}
 a^k_{j}(t)= {S}^{-1}_{ij; k}(t)  d S_{ij; k}(t) + S^{-1}_{ij; k}(t) a^k_{i}(t)S_{ij; k}(t).
 \end{eqnarray}
Moreover, we have
\begin{equation}\label{g.1}
D_{a^{k}_j}=S_{ij; k}^*(D_{a^{k}_i})=S_{ij; k}^{-1}\circ D_{a^{k}_i} \circ S_{ij; k},\quad s^{k}_j=S^{-1}_{ij; k}  s^{k}_i S_{ij;k}+S_{ij;k}^{-1}\circ \frac d{dt} S_{ij;k}.
\end{equation}
Morover,
\begin{equation} \label{g.1.1}
\begin{aligned}
D_{a_j^k}s_j^k
&=D_{a_j^k}\!\left(
S_{ij;k}^{-1}s_i^kS_{ij;k}
+S_{ij;k}^{-1}\frac{\partial S_{ij;k}}{\partial t}
\right) \\
\end{aligned}
\end{equation}
and
\begin{equation} \label{f.1.1}
F_{a^{k}_j}=S_{ij;k}^{-1} F_{a^{k}_i} S_{ij;k}. \end{equation}
We deduce
\begin{equation}
\begin{aligned}
\frac{\partial a_j^k}{\partial t}
&=S_{ij;k}^{-1}\frac{\partial a_i^k}{\partial t}S_{ij;k}
+D_{a_j^k}\!\left(
S_{ij;k}^{-1}\frac{\partial S_{ij;k}}{\partial t}
\right)\\
&=-S_{ij;k}^{-1}(D_{a_i^k}^{*}F_{a_i^k})S_{ij;k}
+S_{ij;k}^{-1}(D_{a_i^k}s_i^k)S_{ij;k}\\
&\qquad+D_{a_j^k}\!\left(
S_{ij;k}^{-1}\frac{\partial S_{ij;k}}{\partial t}
\right)\\
&=-D_{a_j^k}^{*}F_{a_j^k}
+D_{a_j^k}\!\left(
S_{ij;k}^{-1}s_i^kS_{ij;k}
+S_{ij;k}^{-1}\frac{\partial S_{ij;k}}{\partial t}
\right)\\
&=-D_{a_j^k}^{*}F_{a_j^k}+D_{a_j^k}s_j^k,
\end{aligned}
\end{equation}
 which implies that  equation  \eqref{4.2} is   well-defined in $U_i\cap U_j$.

In $U_i\cap U_j$,   we set $H_{ij}:=S_{ij;2}-S_{ij;1}$. Using \eqref{g.1}, we obtain
 \begin{equation}
\frac{\partial H_{ij}}{\partial t}
=H_{ij}s_j^2-s_i^2H_{ij}
+S_{ij;1}\sigma_j-\sigma_iS_{ij;1},\quad H_{ij}(0)=0.
\end{equation}
Since $S_{ij;k}\in G\subset SO(l)$, $H_{ij}^TH_{ij}$ is symmetric. Indeed,
\begin{equation}\label{4.30}
\begin{aligned}
H_{ij}^{T}H_{ij}
&=(S_{ij;2}^{T}-S_{ij;1}^{T})
  (S_{ij;2}-S_{ij;1})\\
&=2I-S_{ij;2}^{T}S_{ij;1}
     -S_{ij;1}^{T}S_{ij;2}.
\end{aligned}
\end{equation}
Moreover, since $s^2_i,s^2_j \in \mathfrak{g}\subset \mathfrak{so}(l)$, we get
\begin{equation}\label{4.31}
    \<H_{ij},H_{ij}s^2_j\>=0, \quad \<s^2_i H_{ij},H_{ij}\>=0.
\end{equation}
Using \eqref{4.30} and \eqref{4.31}, we get
\begin{equation}\label{4.32}
|H_{ij}(x,t)|
\le
\int_0^t
\bigl(
|\sigma_i(x,\tau)|+|\sigma_j(x,\tau)|
\bigr)\,d\tau.
\end{equation}
It follows from    Minkowski’s inequality and Poincare's inequality   that
\begin{equation}\label{4.33}
\begin{aligned}
\|H_{ij}(t)\|_{L^4(U_i\cap U_j)}
&\le
\int_0^t
\left(
\|\sigma_i(\tau)\|_{L^4(U_i\cap U_j)}
+\|\sigma_j(\tau)\|_{L^4(U_i\cap U_j)}
\right)\,d\tau\\
&\le
\int_0^t
\left(
\|\sigma_i(\tau)\|_{L^4(U_i)}
+\|\sigma_j(\tau)\|_{L^4(U_j)}
\right)\,d\tau\\
&\le C\sqrt t\Bigg[
\int_0^t\Big(
\|d\sigma_i(\tau)\|_{L^2(U_i)}^2+\|d\sigma_j(\tau)\|_{L^2(U_j)}^2
\Big)\,d\tau
\Bigg]^{1/2}.
\end{aligned}
\end{equation}
Similarly, we denote $\mathcal H(t):=\sum_{U_i\cap U_j\neq \emptyset}
\|H_{ij}(t)\|_{L^4(U_i\cap U_j)}^2$. Summing \eqref{4.33} over $i$ and $j$, we obtain
\begin{equation}\label{4.34}
\mathcal H(t)
\le Ct\int_0^t\mathcal P(\tau)\,d\tau.
\end{equation}
This implies that
\begin{equation}\label{4.43'}
\begin{aligned}
\int_0^t\Lambda(\tau)\mathcal H(\tau)\,d\tau
&\leq
C\int_0^t\tau\Lambda(\tau)
\left(\int_0^\tau\mathcal P(s)\,ds\right)d\tau\\
&\leq
C\left(\int_0^t\mathcal P(\tau)\,d\tau\right)
\left(\int_0^t\tau\Lambda(\tau)\,d\tau\right).
\end{aligned}
\end{equation}
Using \eqref{f.1.1}, we have
\begin{equation}\label{4.35}
\begin{aligned}
&\quad F_{a_j^2}-F_{a_j^1}\\
&=S_{ij;1}^{-1}
(F_{a_i^2}-F_{a_i^1})S_{ij;1}+S_{ij;2}^{-1}F_{a_i^2}S_{ij;2}
-S_{ij;1}^{-1}F_{a_i^2}S_{ij;1}
\end{aligned}
\end{equation}
and
\begin{equation}\label{4.36}
\begin{aligned}
&\nabla_{a_j^2}F_{a_j^2}
-\nabla_{a_j^1}F_{a_j^1}\\
={}&S_{ij;1}^{-1}
\left(
\nabla_{a_i^2}F_{a_i^2}
-\nabla_{a_i^1}F_{a_i^1}
\right)S_{ij;1}+S_{ij;2}^{-1}
(\nabla_{a_i^2}F_{a_i^2})S_{ij;2}-S_{ij;1}^{-1}
(\nabla_{a_i^2}F_{a_i^2})S_{ij;1}.
\end{aligned}
\end{equation}
Using \eqref{4.35}, we have
\begin{equation}\label{4.37}
\begin{aligned}
\mathcal E(t)&\leq C\sum_i \int_{U_i}|F_{a^2_i}-F_{a^1_i}|^2\phi_i^2 dv+C\sum_{U_i\cap U_j\neq \emptyset}\int_{U_i\cap U_j}|H_{ij}|^2|F_{a^2_i}|^2dv\\
&\leq C\mathcal E_\phi(t)+CK(t)\mathcal H(t).\\
\end{aligned}
\end{equation}
Similarly, using \eqref{4.36}, we get
\begin{equation}\label{4.54}
\begin{aligned}
\mathcal D(t)&\leq C\mathcal D_\phi(t)+C \sum_i \int_{U_i}|b_i|^2|F_{a^2_i}|^2 dv+C\sum_{U_i\cap U_j\neq \emptyset}\int_{U_i\cap U_j}|H_{ij}|^2|\nabla_{a^2_i}F_{a^2_i}|^2 dv\\
&\le C\mathcal D_\phi(t)
+CK(t)\mathcal E(t)+C\Lambda(t)\mathcal H(t).
\end{aligned}
\end{equation}
Integrating \eqref{4.54}  in $t$   and applying \eqref{4.43'}, we obtain
\begin{equation}\label{6.54}
\begin{aligned}
\int_0^t\mathcal D(\tau)\,d\tau
\le{}&
C\int_0^t\mathcal D_\phi(\tau)\,d\tau
+C\int_0^t K(\tau)\mathcal E(\tau)\,d\tau\\
&+C
\left(\int_0^t\tau\Lambda(\tau)\,d\tau\right)
\left(\int_0^t\mathcal P(\tau)\,d\tau\right).
\end{aligned}
\end{equation}
Combining \eqref{6.54} with \eqref{4.18}, we have
\begin{equation}
\begin{aligned}
\int_0^t\mathcal P(\tau)\,d\tau
\le{}&
C\int_0^t\mathcal D_\phi(\tau)\,d\tau
+C\int_0^t h(\tau)\mathcal E(\tau)\,d\tau\\
&+C
\left(\int_0^t\tau\Lambda(\tau)\,d\tau\right)
\left(\int_0^t\mathcal P(\tau)\,d\tau\right).
\end{aligned}
\end{equation}
Since $\tau\Lambda(\tau)\in L^1(0,T)$, it follows from the Lebesgue continuity that
\begin{equation}\label{6.57}
\begin{aligned}
\int_0^t\mathcal P(\tau)\,d\tau
\le{}&
C\int_0^t\mathcal D_\phi(\tau)\,d\tau+C\int_0^t h(\tau)\mathcal E(\tau)\,d\tau
\end{aligned}
\end{equation}
for $t>0$ sufficiently small. Substituting \eqref{6.57} into the interior estimate \eqref{4.24}, we obtain
\begin{equation}\label{4.42}
\mathcal E_\phi(t)
+\int_0^t\mathcal D_\phi(\tau)\,d\tau
\le C\int_0^t h(\tau)\mathcal E(\tau)\,d\tau.
\end{equation}
It follows that
\begin{equation}\label{4.59}
\int_0^t\mathcal P(\tau)\,d\tau
\le C\int_0^t h(\tau)\mathcal E(\tau)\,d\tau.
\end{equation}
Combining \eqref{4.59}   with \eqref{4.34}  and \eqref{4.42}, we have
\begin{equation}
\mathcal E(t)
\le C\left(1+tK(t)\right)
\int_0^t h(\tau)\mathcal E(\tau)\,d\tau.
\end{equation}
Since $h(t)\in L^1[0,T]$, $tK(t)\in L^\infty(0,T)$ and $\mathcal E(0)=0$, it follows from Gronwall’s inequality that
$
\mathcal E(t)=0.$ Then,  using \eqref{4.13} and \eqref{4.23.1}, we have $b_i=0$ and $\sigma_i=0$ in each $U_i$, so we prove uniqueness of the pair $(a^k_i,s^k_i)$.

Finally, we recover uniqueness of weak solutions  $A^k(t)$ of the Yang-Mills flow. Since $\sigma_i=0$, we obtain
\begin{equation}
\begin{aligned}
\frac{\partial}{\partial t}(S_i^2-S_i^1)
&=(S_i^2-S_i^1)s_i^1
  +S_i^2\sigma_i\\
&=(S_i^2-S_i^1)s_i^1,
\end{aligned}
\end{equation}
with $S^2_i(0)=S^1_i(0)$. By the same linear algebra argument as one in \eqref{4.31}, we have
\begin{equation}
\begin{aligned}
\frac12\frac{d}{dt}
\|S_i^2-S_i^1\|_{L^2(U_i)}^2
&=\int_{U_i}
 \!\left<
S_i^2-S_i^1, \, (S_i^2-S_i^1)s_i^1
\right>\,dv=0.
\end{aligned}
\end{equation}
This implies that $S^1_i(t)=S^2_i(t)$ and we conclude that $A^1(t)=A^2(t)$ in each $U_i$.
\end{proof}


\begin{thebibliography}{}

\bibitem{AB}
 M. Atiyah and R. Bott,
    \newblock{The Yang-Mills equations over Riemann
surfaces},
    \newblock{Phil. Trans. Roy. Soc. London A}, {\bf   308} (1982),   524--615.




\bibitem{CCH} J, Casey,  C. H. Chan and M.-C. Hong,      \newblock{The Yang-Mills flow in   vector bundles over higher dimensional Riemannian manifolds},     \newblock{Preprint in 2026: arXiv:2610.38124}.



\bibitem{D1} S. K. Donaldson,       \newblock{An application of gauge
theory to four dimensional topology},     \newblock{J. Differential Geom. } {\bf 18}
(1983)  269-278.

    \bibitem{D2}
 S. K. Donaldson,
    \newblock{Anti-self-dual
Yang-Mills connections on complex algebraic surfaces and stable
vector bundles},
    \newblock{Proc. Lond. Math. Soc.}  {\bf 50} (1985)    1-26.









\bibitem{F} K. O. Friedrichs,
    \newblock{Differential forms on Riemannian manifolds},
    \newblock{Comm. Pure Appl. Math. } {\bf 8} (1955)     551–590.





\bibitem{G1}
M.   Gaffney,
    \newblock{The harmonic operator for exterior differential forms},
    \newblock{Proceedings of the National
Science, U.S.A., } {\bf 37} (1951)     48-50.
\bibitem{Ga}
A. Gastel,
Nonuniqueness for the Yang--Mills heat flow,
J. Differential Equations \textbf{187} (2003), no. 2, 391--411.

\bibitem{HT1}
 M.-C. Hong and G. Tian,
    \newblock{Asymptotical behaviour of the Yang-Mills
flow and singular Yang-Mills connections},
    \newblock{Math. Ann.}  {\bf 330} (2004)   441--472.





    \bibitem{HTY} M.-C. Hong,  G. Tian   and H. Yin, \newblock {The Yang-Mills $\alpha$-flow in vector bundles over four manifolds and its applications,} \newblock  {Comm. Math. Helv.}    {\bf 90} (2015)
    75--120.

\bibitem{ISS}
T.~Iwaniec, C.~Scott, and B.~Stroffolini,
Nonlinear Hodge theory on manifolds with boundary,
Ann. Mat. Pura Appl. (4) 177 (1999), 37--115.
\bibitem{LSU}
     O. A. Ladyzhenskaya, V. A. Solonnikov and N. N.
Ural'ceva,
    \newblock{Linear and qusilinear equations of
parabolic type},
    \newblock{Tanslations of Mathematical
Monographs 23. Providence, Rhode Island: American Mathematical
Society} 1968.





 \bibitem{Mo}
    C. B. Morrey,
        \newblock{ Multiple integrals in the calculus of variations,}
        \newblock{Springer-Verlag, Berlin, 1966.}


     \bibitem{Ra}
 J. Rade,
    \newblock{On the Yang-Mills heat equation in two and three dimensions},
    \newblock{J. Reine Angew. Math.}    {\bf 431 }  (1992)  123-163.


\bibitem{R} T. Rivi\`{e}re, \newblock {Interpolation spaces and energy quantization for Yang-Mills fields}, \newblock  Comm. Anal. Geom. {\bf 10} (2002) 683-708.





\bibitem{Se}
 S. Sedlacek,
    \newblock{The Yang-Mills functional over $4$-manifolds},
    \newblock{Commun. Math. Phys.}, {\bf    86} (1982)    515-527.


\bibitem{Si}
  C. T. Simpson,
    \newblock{Constructing variations of Hodge structure
using Yang-Mills theory and applications to uniformization},
    \newblock{J. Amer. Math. Soc.}  {\bf   1} (1988)   867-918.

 \bibitem{Sib}
B. Sibley,
    \newblock{Asymptotics of the Yang-Mills flow for holomorphic vector bundles over K\"ahler manifolds: The canonical
structure of the limit.},
    \newblock{J. Reine Angew. Math.}    {\bf 706}  (2015)  123-191.







\bibitem{St2}
 M. Struwe,
    \newblock{The Yang-Mills flow in four
dimensions:},
    \newblock{Calc. Var. $\&$ PDEs}, {\bf 2} (1994)   123--150.






\bibitem{TTi}
 T. Tao and G. Tian,
    \newblock{A singularity
removal theorem for Yang-Mills fields in higher dimensions},
    \newblock{J. Amer. Math. Soc.}, {\bf   17} (2004)    557--593.

\bibitem{Ti}
 G. Tian,
    \newblock{Gauge theory and calibrated geometry, I},
    \newblock{Ann.  Math.}, {\bf 151} (2000)  193--268.

\bibitem{U1}
 K. Uhlenbeck,
    \newblock{Removable singularities in
Yang-Mills fields},
    \newblock{Commun. Math. Phys.}, {\bf 83} (1982)    11--30.


    \bibitem{U2}
K. Uhlenbeck,
    \newblock{Connections with $L^p$ bounds on
curvature},
    \newblock{ Commun. Math. Phys.}, {\bf 83 } (1982)  31--42.



 \bibitem{W}
A.  Waldron,
    \newblock{Instantons and  singularities in
Yang-Mills flow},
     \newblock{Calc. Var. $\&$ PDEs}  {\bf   55}        (2016)  no. 5, Art. 113, 31pp.



\end{thebibliography}
\end{document}